\documentclass[dvips,ssy,preprint]{imsart}

\RequirePackage[OT1]{fontenc}
\RequirePackage{amsthm,amsmath,mathrsfs}
\RequirePackage[numbers]{natbib}
\RequirePackage[colorlinks,citecolor=blue,urlcolor=blue]{hyperref}
\RequirePackage[normalem]{ulem}

\doi{10.1214/12-SSY64}
\pubyear{2012}
\volume{2}
\issue{0}
\firstpage{1}
\lastpage{59}

\startlocaldefs
\numberwithin{equation}{section}
\theoremstyle{plain}
\newtheorem{thm}{Theorem}[section]
\newtheorem{lemma}{Lemma}[section]
\newtheorem{corollary}{Corollary}[section]

\theoremstyle{definition}
\newtheorem*{Remarks}{Remarks}
\newtheorem*{Remark}{Remark}
\endlocaldefs

\newcommand{\ra}{\rightarrow}

\newcommand{\ZA}{\mathcal{Z}}

\newcommand{\A}{\mathcal{A}}
\newcommand{\MA}{\mathcal{M}}
\newcommand{\EA}{\mathcal{E}}
\newcommand{\XA}{\mathcal{X}}

\usepackage{amssymb}

\begin{document}

\begin{frontmatter}
\title{Asymptotics of the invariant measure in mean field models with jumps\protect\thanksref{T1}}
\runtitle{Invariant measure in mean field models}
\thankstext{T1}{A part of this paper without proofs was presented as an invited paper at the 2011 Annual Allerton Conference on Communication, Control, and Computing, Allerton, IL, USA, September 2011.}

\begin{aug}
\author[a]{\fnms{Vivek S.} \snm{Borkar}\thanksref{t1}\ead[label=e1]{borkar.vs@gmail.com}}
\and
\author[b]{\fnms{Rajesh} \snm{Sundaresan}\thanksref{t2}\ead[label=e2]{rajeshs@ece.iisc.ernet.in}}

\address[a]{Department of Electrical Engineering \\
Indian Institute of Technology Bombay\\
Powai, Mumbai 400076, India\\
\printead{e1}}

\address[b]{Department of Electrical\\\quad  Communications Engineering\\
Indian Institute of Science\\
Bangalore 560012, India\\
\printead{e2}}

\thankstext{t1}{Work supported by a J.\ C.\ Bose Fellowship and by IFCPAR (Indo-French Centre for the Promotion of Advanced Research), Project 4000-IT-1.}
\thankstext{t2}{Work supported by IFCPAR (Indo-French Centre for the Promotion of Advanced Research), Project 4000-IT-1.}
\runauthor{V. S. Borkar and R. Sundaresan}

\affiliation{Indian Institute of Technology Bombay and Indian Institute of Science}
\end{aug}

\begin{abstract}
We consider the asymptotics of the invariant measure for the process of
the empirical spatial distribution of $N$ coupled Markov chains in the
limit of a large number of chains. Each chain reflects the stochastic
evolution of one particle. The chains are coupled through the
dependence of the transition rates on this spatial distribution of
particles in the various states. Our model is a caricature for medium
access interactions in wireless local area networks. It is also
applicable to the study of spread of epidemics in a network. The
limiting process satisfies a deterministic ordinary differential
equation called the McKean-Vlasov equation. When this differential
equation has a unique globally asymptotically stable equilibrium, the
spatial distribution asymptotically concentrates on this equilibrium.
More generally, its limit points are supported on a subset of the
$\omega$-limit sets of the McKean-Vlasov equation. Using a
control-theoretic approach, we examine the question of large deviations
of the invariant measure from this limit.
\end{abstract}

\begin{keyword}[class=AMS]
\kwd[Primary ]{60K35}
\kwd{60F10}
\kwd{68M20}
\kwd{90B18}
\kwd{49J15}
\kwd{34H05}.
\end{keyword}

\begin{keyword}
\kwd{Decoupling approximation}
\kwd{fluid limit}
\kwd{invariant measure}
\kwd{McKean-Vlasov equation}
\kwd{mean field limit}
\kwd{small noise limit}
\kwd{stationary measure}
\kwd{stochastic Liouville equation}.
\end{keyword}

\received{\smonth{1} \syear{2012}}

\end{frontmatter}

\section{Introduction}

Spurred by the seminal work of Bianchi \cite{Bianchi}, there has been a flurry of activity in the communication networks community on mean field models for carrier sense multiple access (CSMA) protocols and their large time behavior. The continuous-time model for the wireless local area network (WLAN) is as follows. There are $N$ particles (nodes) in the network. At each instant of time, a particle's state is a particular value taken from the finite state space $\ZA = \{ 0, 1, \ldots, r-1 \}$. A particle's state represents the number of failed attempts at transmission of the head-of-the-line packet at that particle's queue. When a particle is in state $i$, a successful transmission gets the packet out of the system, and the particle moves to state 0 to service the next packet. A failed transmission moves the particle to state $i+1$ (mod $r$). In the case when $i$ was initially $r-1$, that is, $r-1$ unsuccessful transmission attempts were already made, another failed attempt results in the discarding of the packet. The particle then moves to state 0 with the next packet readied for transmission. We may interpret $r$ as the maximum number of transmission attempts. The transition rate for a particle from state $i$ to state $j$ is governed by {\em mean field dynamics}, that is, the transition rate is $\lambda_{i,j}(\mu_N(t))$ where $\mu_N(t)$ is the empirical distribution of the states of particles at time $t$. If $X^{(N)}_n(t)$ is the state of the $n$th particle at time $t$, then one may write $\mu_N(t)$ as
\[
  \mu_N(t) = \frac{1}{N} \sum_{n=1}^N \delta_{ \{X_n^{(N)}(t) \}}.
\]
The particles interact only through the dependence of their transition rates on the current empirical measure $\mu_N(t)$.

The transitions allowed in the above model are from state $i$ to either $i+1$ (mod $r$) or 0. Let us say that $\EA$ denotes the set of allowed transitions. In the above model,
\[
  \EA \hspace*{-.03in} = \hspace*{-.03in} \{ (i,i+1), i = 0,1,\ldots, r-1\} \cup \{ (i,0), i = 0,1,\ldots, r-1\}
\]
where the addition is taken modulo $r$.

The process $X^{(N)}(\cdot) = \{ X^{(N)}_n(\cdot), 1 \leq n \leq N \}$ is clearly a Markov process. But one difficulty needs to be surmounted in analyzing this system: the size of the state space grows exponentially in the number of particles. A step towards addressing this difficulty is to consider the evolution or flow of the empirical measure over time, which we shall call {\em empirical process}. This is a stochastic Liouville equation that lives on a smaller state space. In the infinite particle limit, this evolution turns out to be deterministic and is given by the McKean-Vlasov equation, whose large time behavior is an indicator of what one might expect of a finite but large population. In particular, if the deterministic evolution, given by the McKean-Vlasov equation, has a unique globally asymptotically stable equilibrium, then the states of a finite number of tagged particles are asymptotically independent and their joint law is  given by the product of this equilibrium measure. The idea in fact was introduced by Kac as a simple model in kinetic theory \cite{Kac} and was later studied by McKean and others (see, e.g., \cite{McKean}). See \cite{Sznitman} for an extensive account and \cite{Graham} for a treatment of processes with jumps.

Several papers have provided rigorous analyses, along the above lines, of Bianchi's heuristic for studying WLANs. See, e.g., \cite{Bordenave-Allerton,Sharma-Ganesh-Key,McDonald,Sukhov,Benaim,Ramaiyan}. See \cite{Duffy} for an excellent survey, \cite{Stolyar,Anantharam,Anantharam-Benchekroun} for early precursors, and \cite{Benaim-Weibull} for an application of the same technique in game theory. As remarked in \cite{Duffy}, experimental evidence for CSMA protocols indicates that the model and the predictions made by the analyses are surprisingly accurate even for small populations. Indeed, this is one of the main reasons for the model's enormous popularity. Is there a justification for this concentration phenomenon?

The mean-field analysis has been successful when the McKean-Vlasov equation has a unique globally asymptotically stable equilibrium. This is indeed the case in the simplest of WLAN settings with exponential backoff parameters. But there are settings with multiple stable equilibria \cite{Sukhov} or with a unique equilibrium that is not globally stable (see \cite{Benaim} for a malware propagation model). In both cases, the dynamics governed by the McKean-Vlasov equation has multiple $\omega$-limit sets. Significant effort has gone into identifying sufficient conditions for a unique equilibrium \cite{Ramaiyan}, and into identifying further sufficient conditions for a unique globally asymptotically stable equilibrium \cite{McDonald}. If there are multiple $\omega$-limit sets for the McKean-Vlasov dynamics, which of these characterize the limiting behavior?

As a step in the direction of understanding these questions, we study the continuous-time model in this paper with the following goals.
\begin{enumerate}
  \item Obtain a large deviation principle over finite time durations for the sequence of empirical measures and empirical processes, uniformly in the initial condition. This is of course rough asymptotics for large $N$, but suggests exponentially fast convergence to the deterministic limit.
  \item Obtain a large deviation principle for the sequence of invariant measures, with single or multiple stable limit sets. This helps resolve which of the several $\omega$-limit sets, when there are several, will be selected in the large $N$ limit. (Our work can be straightforwardly extended to study the nature of transitions and exit times from the neighborhood of one stable equilibrium to the neighborhood of another, and help understand metastable behavior in such systems. We do not pursue these here.)
  \item Provide a control theoretic framework to solve the problem of invariant measure in order to expose its strength and limitation. As we will see, we can go quite some distance using this {\em deterministic} approach, but eventually need to study the noisy system for resolution of some degeneracies.
\end{enumerate}
\eject

It must be noted that the above continuous-time model does not perfectly capture all aspects in a WLAN. In particular, interactions and changes of states occur in discrete-time units of slots in WLANs, and multiple nodes may transit in one slot. Multiple transitions never occur, almost surely, in our continuous-time model. Nevertheless, if the discrete-time model's transition rates and the slot sizes are appropriately scaled down as $N$ grows so that the transition rates approach constants, our continuous time model provides accurate predictions of behavior on the discrete-time model. Our model also has wider applicability, one example being the study of spread of epidemics in a network; see \cite[Sec. 2.4]{DjehicheKaj}.

Large deviation principles over compact time durations for interacting
diffusions and interacting jump Markov processes have been well-studied
by several authors, e.g.,
\cite{Dawson-Gartner,Leonard,DjehicheKaj,Feng-empiricalMeasure,DaiPra,DelMoral-Zajic}.
The works \cite{DjehicheKaj,DaiPra,DelMoral-Zajic}
establish large deviation principles for empirical measures in path
space over finite time durations and characterize the rate functions,
while \cite{Feng-empiricalMeasure} considers infinite time durations
under the inductive topology. The works \cite{Dawson-Gartner,Leonard,DjehicheKaj,DaiPra}, and
\cite{Feng-empiricalFlow} also study large deviation of the empirical
process from the McKean-Vlasov limit, over finite time durations or
infinite time durations under the inductive topology
(\cite{Feng-empiricalFlow}). The rate function measures the difficulty
of passage of the empirical process in the neighborhood of a deviating
path.

For a fixed $N$, when $t \ra +\infty$, the stationary or invariant
measure is of interest. When the limiting McKean-Vlasov dynamics has a
unique globally asymptotically stable equilibrium, the invariant
measure converges weakly in the infinite particle limit to the point
mass at the equilibrium of the McKean-Vlasov dynamics (see, e.g.,
\cite{Benaim}). When there are multiple equilibria, the invariant
measure concentrates on a subset of the $\omega$-limit sets for the
McKean-Vlasov dynamics (\cite[Ch. 6]{Freidlin} and
\cite{Benaim-Weibull}). Large deviations from this limit have been
well-studied (see \cite[Ch. 6]{Freidlin}, \cite{Shwartz}). As indicated
earlier, large deviation results for the invariant measure help resolve
which of several possible $\omega$-limit sets may be selected in the
limit of a large number of nodes. They also help understand and predict
metastable behavior in such systems. If the system is trapped in an
undesirable equilibrium, exit times from domains and likely paths can
be predicted.

Our approach for solving the large deviations of the invariant measure exploits a control-theoretic view described in \cite{Biswas}, which considers a class of diffusions and generalizes results of \cite{Sheu-diffusion} and \cite{Day} on small noise asymptotics of invariant measures and exit probabilities in diffusions. See \cite{Sheu-jump} for small noise asymptotics of exit probabilities in processes with jumps. The control-theoretic approach differs from those of Freidlin and Wentzell~\cite{Freidlin} and Shwartz and Weiss \cite{Shwartz}. The latter rely on a study of an embedded Markov chain of states at hitting times of neighborhood of the stable limit sets (see, e.g., \cite[Ch. 6.4]{Freidlin}). Our control-theoretic approach enables the identification of a unique rate function when there is a unique globally asymptotically stable equilibrium for the McKean-Vlasov dynamics. When dealing with multiple $\omega$-limit sets though, our approach eventually requires the study of the noisy system associated with the embedded Markov chain described above, for resolution of certain boundary conditions needed for a full characterization of the rate function.

We now outline our main arguments and describe the paper's organization.

We begin with a formal description of the model and statements of the
main results on invariant measures in Section \ref{sec:model+results}.
In Section \ref{sec:empirical-measure}, we first establish a large
deviation principle for empirical measures of paths over finite
durations. For pure jump processes with interactions, this was
established by \cite{Leonard,DjehicheKaj,Feng-empiricalMeasure,DaiPra,DelMoral-Zajic}, as
indicated earlier. While \cite{Leonard} considers a fixed initial
condition for all the particles, \cite{DaiPra,Feng-empiricalMeasure}, and \cite{DelMoral-Zajic} consider
random, independent, and identically distributed (also called chaotic)
initial conditions for the particles. To establish the large deviation
principle for the interacting system, they first establish the large
deviation principle for the independent and noninteracting system via
Sanov's theorem, then exploit the Girsanov transformation to describe
the probability measure for the interacting case, and then apply the
Laplace-Varadhan principle. In order to eventually pass to the
invariant measure, we need to establish a stronger uniform large
deviation principle when the initial conditions of the particles are
such that the initial empirical measures converge weakly, but are
otherwise arbitrary. (See the remark following \cite[Th.
4.1]{Feng-empiricalMeasure}.) The limiting (initial) empirical measure
defines the initial condition for the limiting deterministic
McKean-Vlasov dynamics. This stronger uniform large deviation principle
alluded to above is available for diffusions with mean field
interactions in \cite{Dawson-Gartner} and for certain classes of jump
processes in \cite{DjehicheKaj} where the holding times alone are
mediated by the interaction and not the jump probabilities. Our result
is a mild extension facilitated by a generalization of Sanov's theorem
given in \cite{Dawson-Gartner}. While this part of our result is not
surprising, we could not find a ready reference in the literature, and
so we state the result in Section \ref{sec:empirical-measure} and
provide a proof in Section \ref{sec:Proof-empirical-measure} based on
the approach in \cite{Leonard}. We reemphasize that this result is only
for finite durations, and is only a step towards addressing our next
goal of asymptotics of the invariant measure.

In Section \ref{sec:empirical-flow}, we apply the contraction principle to obtain results on a large deviation from the McKean-Vlasov limit. As in \cite{Leonard}, we prove this under the finer uniform norm topology. In Section \ref{sec:empirical-flow-initial-terminal-time}, we once again apply the contraction principle to argue a large deviation principle for terminal measure. We then establish some crucial estimates on the rate function for use in later sections. Finally in this section, we argue that if the initial measures satisfy a large deviation principle, then so do the joint initial and terminal measures.

In Section \ref{sec:invariant-measure}, we prove the large deviation principle for the invariant measure. To do this, we first establish a subsequential large deviation principle, then argue that the rate function satisfies the dynamic programming equation for a particular control problem, and then finally show that the rate function is unique if the associated McKean-Vlasov dynamics has a unique globally asymptotically stable equilibrium. The arguments to establish the uniqueness of the rate function parallel those of \cite{Biswas} with the significant difference that while the dynamic programming equation was arrived at in \cite{Biswas} via the theory of viscosity solutions, here we use the contraction principle. While this provides a complete result in case of a single globally asymptotically stable equilibrium, it leaves some indeterminacy regarding uniqueness of the so called `potential function' whose global minima the invariant measure concentrates on in the large $N$ limit. To resolve this, one has to fall back upon the framework of Freidlin and Wentzell  \cite[Ch.6]{Freidlin} with minor modifications. These modifications are detailed in the Appendix.

Proofs that are not central to our control-theoretic view point are relegated to Sections \ref{sec:Proof-empirical-measure}, \ref{sec:ProofOfS_T-bounding}, \ref{sec:uniform-continuity}, and \ref{sec:ProofJointLDP}. The Appendix details the minor modifications to the arguments of Freidlin and Wentzell \cite[Ch.6]{Freidlin} for resolution of the rate function values at the stable limit points.

\section{The model and main results}
\label{sec:model+results}

Consider $N$ interacting Markov chains denoted by
\[
  X^{(N)}_n(t), ~1 \leq n \leq N, ~t \geq 0,
\]
on a finite state space $\ZA = \{0, \ldots, r-1\}$, with dynamics as follows. $X^{(N)}_n(t)$ denotes the state of the $n$th particle at time $t$. Let $\mu_N(t)$ be the empirical measure of the particles at time $t$, that is,
\[
  \mu_N(t) = \frac{1}{N} \sum_{n=1}^N \delta_{\{X^{(N)}_n(t)\}} \in \MA_1(\ZA),
\]
where $\MA_1(\ZA)$ is the set of probability vectors over $\ZA$, endowed with the topology of weak convergence on $\MA_1(\ZA)$. A particle $n$ in state $i$ transits to state $j \in \ZA$ with a rate $\lambda_{i,j}(\mu_N(t))$ that depends on the states of the other particles only through the empirical measure at that time. Thus the processes interact only through the dependence of their transition rates on the current empirical measure.

Let $\EA \subset \ZA \times \ZA \setminus \{ (i,i) ~|~ i \in \ZA \}$ be the set of admissible jumps for each particle. Thus $\lambda_{i,j}(\cdot) \equiv 0$ whenever $(i,j) \notin \EA$ and $i \neq j$. We write $\ZA_i = \{ j \in \ZA ~|~ (i,j) \in \EA \}$ for the set of states to which a particle can jump from state $i$.

We make the following assumptions throughout the paper.
\begin{itemize}
  \item[({\bf A1})] The graph with vertices $\ZA$ and directed edges $\EA$ is irreducible.
  \item[({\bf A2})] The mappings $\mu \in \MA_1(\ZA) \mapsto \lambda_{i,j}(\mu) \in [0, +\infty)$ are Lipschitz.
  \item[({\bf A3})] The rates of the admissible jumps are uniformly bounded away from zero, that is, there exists $c > 0$ such that, for all $\mu \in \MA_1(\ZA)$ and all $(i,j) \in \EA$, we have $\lambda_{i,j}(\mu) \geq c$.
\end{itemize}
Since $\MA_1(\ZA)$ is compact and $\lambda_{i,j}(\cdot)$ are continuous by assumption ({\bf A2}), the rates are uniformly bounded from above, that is, there is a $C < \infty$ such that for all $\mu \in \MA_1(\ZA)$, and all $(i,j) \in \EA$, we have $\lambda_{i,j}(\mu) \leq C$.

For any $T \in ( 0, +\infty)$, write $X^{(N)}_n : [0,T] \rightarrow \ZA$ for the process of evolution of particle $n$ over time. This is an element of the set $D([0,T], \ZA)$ of all cadlag paths from $[0,T]$ to $\ZA$ equipped with the Skorohod topology. We set the path to be left continuous at $T$. Let
\[
  X^N = (X^{(N)}_n, 1 \leq n \leq N) \in D([0,T], \ZA^N)
\]
denote the full description of paths of all $N$ particles. The initial condition at time $t = 0$ is $z^N = (z_n, 1 \leq n \leq N)$. The process $X^N$, with its law denoted $\mathbb{P}^{(N)}_{z^N}$, is a Markov process with cadlag paths, state space $\ZA^N$, and generator $\mathscr{A}^{(N)}$ acting on bounded measurable functions $\Phi$ according to
\begin{equation}
  \label{eqn:n-particle-generator}
  \mathscr{A}^{(N)} \Phi(a^N) = \sum_{n=1}^N \sum_{j \in \ZA_{a_n}}
                  \left[ \lambda_{a_n, j} \left( g_N(a^N) \right) \right]
                  \left( \Phi ( \sharp(a^N, n, j) ) - \Phi(a^N) \right)
\end{equation}
where $a^N = (a_n, 1 \leq n \leq N) \in \ZA^N$, $\sharp(a^N,n,j)$ is the element of $\ZA^N$ that results from replacing the $n$th component of $a^N$ with $j$, and $g_N(a^N) = \frac{1}{N} \sum_{n'=1}^N \delta_{a_{n'}}$ is the empirical measure associated with the configuration $a^N$. (The upper boundedness of the rates implies that the martingale problem associated with the generator $\mathscr{A}^{(N)}$, operating on bounded measurable functions, and the initial condition $z^N \in \ZA^N$ admits a unique solution $\mathbb{P}^{(N)}_{z^N} \in \MA_1(D([0,T], \ZA^N))$; see for e.g., \cite[Problem 4.11.15, p. 263]{Ethier}).

The empirical measure process associated with $X^N$ is given by
\[
  \mu_N : t \in [0,T] \mapsto \mu_N(t) = \frac{1}{N} \sum_{n=1}^N \delta_{\{X^{(N)}_n(t)\}} \in \MA_1^{(N)}(\ZA),
\]
where $\MA_1^{(N)}(\ZA) = \{ g_N(a^N) ~|~ a^N \in \ZA^N \} \subset \MA_1(\ZA)$. As a consequence of (\ref{eqn:n-particle-generator}), the process $\mu_N$ is itself a Markov process with cadlag paths, finite state space $\MA_1^{(N)}(\ZA)$, and generator $\mathcal{A}^{(N)}$ acting on bounded measurable functions $\Psi$ according to
\begin{equation}
  \label{eqn:flow-generator}
  \mathcal{A}^{(N)} \Psi( \xi ) = N \sum_{i \in \ZA} \sum_{j \in \ZA_i }
                  \left[ \xi(i) \lambda_{i, j} \left( \xi \right) \right]
                  \left( \Psi ( \xi - N^{-1} \delta_i + N^{-1} \delta_j ) - \Psi(\xi) \right).
\end{equation}

Write $\lambda_{i,i}(\xi) = - \sum_{j' \neq i} \lambda_{i,j'} (\xi) $ and define
\begin{equation}
  \label{eqn:rate-matrix}
  A_{\xi} = (\lambda_{i,j}(\xi))_{(i,j) \in \ZA \times \ZA}
\end{equation}
to be the {\em rate matrix} over $\ZA$. It is well known that the family $(\mu_N, N \geq 1)$ satisfies the weak law of large numbers in the following sense: if $\mu_N(0) \rightarrow \nu$ weakly as $N \rightarrow \infty$ for some $\nu \in \MA_1(\ZA)$, then $\mu_N \rightarrow \mu$ uniformly on compacts in probability, where $\mu$ solves the {\em McKean-Vlasov equation}
\begin{equation}
  \label{eqn:mckvla-dynamics}
  \dot{\mu}(t) = A_{\mu(t)}^* \mu(t)
\end{equation}
with initial condition $\mu(0) = \nu$. Here $\mu(t)$ is interpreted as a column vector and $A_{\mu(t)}^*$ is the adjoint/transpose of the matrix $A_{\mu(t)}$. By assumption ({\bf A2}), uniqueness of solutions holds for the nonlinear ordinary differential equation (ODE) (\ref{eqn:mckvla-dynamics}).

As a consequence of the assumption ({\bf A1}), for each $N \geq 1$, there is a unique invariant measure for the Markov process $X^N$, and hence there is a unique invariant measure $\wp^{(N)}$ for the $\MA_1^{(N)}(\ZA)$-valued Markov process $\mu_N$.

Define $\tau : \mathbb{R} \rightarrow \mathbb{R}_+$ to be
\begin{equation}
  \label{eqn:tau}
  \tau(u) = e^u - u - 1,
\end{equation}
and let $\tau^*: \mathbb{R} \rightarrow \overline{\mathbb{R}}_+$ be its Legendre conjugate
\begin{equation}
  \label{eqn:tau*}
  \tau^*(u) = \left\{
    \begin{array}{ll}
      (u+1) \log (u+1) - u & \mbox{ if } u > -1 \\
      1                    & \mbox{ if } u = -1 \\
      +\infty              & \mbox{ if } u < -1.
    \end{array}
  \right.
\end{equation}

The first main result is on the large $N$ asymptotics of the sequence of invariant measures $(\wp^{(N)}, N \geq 1)$.

\begin{thm}
  \label{thm:invariant}
  Assume {\em ({\bf A1})} - {\em ({\bf A3})}. Let the McKean-Vlasov equation $\dot{\mu}(t) = A_{\mu(t)}^* \mu(t)$ have a unique globally asymptotically stable equilibrium $\xi_0$. Then the sequence $(\wp^{(N)}, N \geq 1)$ satisfies the large deviation principle with speed $N$ and good rate function $s$ given by
  \begin{equation}
  \label{eqn:s-2}
    s(\xi) = \inf_{\hat{\mu}} \int_{[0, +\infty)} \Big[ \sum_{(i,j) \in \EA} (\hat{\mu}(t)(i)) \lambda_{i,j}(\hat{\mu}(t))
             \tau^* \left( \frac{\hat{l}_{i,j}(t)}{\lambda_{i,j}(\hat{\mu}(t))} - 1 \right) \Big] ~dt
  \end{equation}
where the infimum is over all $\hat{\mu}$ that are solutions to the dynamical system $\dot{\hat{\mu}}(t) = -\hat{L}(t)^* \hat{\mu}(t)$ for some family of rate matrices $\hat{L}(\cdot)$, with initial condition $\hat{\mu}(0) = \xi$, terminal condition $\lim_{t \ra +\infty} \hat{\mu}(t) = \xi_0$, and $\hat{\mu}(t) \in \MA_1(\ZA)$ for all $t \geq 0$.
\end{thm}

\begin{Remarks}
 1. As we shall see later, the dynamics $\dot{\hat{\mu}}(t) = -\hat{L}(t)^* \hat{\mu}(t)$ corresponds to a time reversal when compared with the direction of the McKean-Vlasov dynamics. The rate function is given by the cost, associated with a certain control problem, of the cheapest path (across all time) that transports the system state from the globally asymptotically stable equilibrium $\xi_0$ to $\xi$, in the forward-time dynamics proceeding in the direction of the McKean-Vlasov dynamics. In the time-reversed dynamics, this is the cost of a path with initial state $\xi$ and terminal state $\xi_0$.

2. By setting $\hat{l}_{i,j}(t) \equiv \lambda_{i,j}(\xi_0)$, we see that if $\hat{\mu}(0) = \xi_0$, then the system state remains $\hat{\mu}(t) \equiv \xi_0$, and the integral is 0. It follows that $s(\xi_0) = 0$.

3. In reversed time, the system state under the dynamics $\dot{\hat{\mu}}(t) = -\hat{L}(t)^* \hat{\mu}(t)$ may not in general lie in $\MA_1(\ZA)$. The minimization however is over all paths that are constrained to lie in $\MA_1(\ZA)$. See additional remarks after Lemma~\ref{lem:one-trajectory}.

4. We now make some remarks on why Theorem \ref{thm:invariant} and the
soon to \mbox{follow} Theorem \ref{thm:invariant-multiple-equilibria}
are not subsumed by the works of \cite{Shwartz} and \cite{Freidlin}.
The \mbox{transition} rates for the Markov process $\mu_N(\cdot)$ are
$(\mu_N(t)(i)) \lambda_{i,j}(\mu_N(t))$. While
$\lambda_{i,j}(\mu_N(t))$ is indeed bounded away from zero when $(i,j)
\in \EA$, $(\mu_N(t)(i)) \lambda_{i,j}(\mu_N(t))$ is not, and so the
logarithm of these rates is not bounded, a requirement in
\cite{Shwartz}. Next, the process $\mu_N$ takes values only in
$\MA_1(\ZA)$, and so when at the boundary $\MA_1(\ZA)$, it is
constrained to move only in those directions that keep it within
$\MA_1(\ZA)$. As a consequence, a finiteness condition \cite[p.146,
II]{Freidlin} on the associated Lagrangian function does not hold.

5. Theorem \ref{thm:invariant} provides a complete characterization of the rate function. However, numerical computation of the rate function is a challenging problem. One might possibly discretize time and the state space, and employ dynamic programming techniques to get an approximation. This is an interesting line of work that is beyond the scope of this paper. See \cite{Benaim-Weibull} and references therein for some results based on exit times.
\end{Remarks}


We next state a generalization of Theorem \ref{thm:invariant} when there may be multiple $\omega$-limit sets. Consider a time-varying rate matrix $L(t)$ associated with the time-varying rates $(l_{i,j}(t), (i,j) \in \EA)$ and let $L(t)^*$ be its adjoint. Write $\mu$ for the solution to the dynamical system $\dot{\mu}(t) = L(t)^* \mu(t)$ with initial condition $\mu(0) = \nu$. Define
  \begin{equation}
    \label{eqn:finite-rate-evaluation}
    S_{[0,T]}(\mu | \nu) = \int_{[0,T]} \Big[ \sum_{(i,j) \in \EA} (\mu(t)(i)) \lambda_{i,j}(\mu(t))
                           \tau^* \left( \frac{l_{i,j}(t)}{\lambda_{i,j}(\mu(t))} - 1 \right) \Big] ~dt.
  \end{equation}
As we shall see later, this is the cost of moving the system state along the trajectory $\mu$ with initial state $\mu(0) = \nu$. Let $V(\xi | \nu)$ be the so-called quasipotential defined by
  \begin{equation}
    \label{eqn:quasipotential}
    V( \xi | \nu) = \inf \{ S_{[0,T]} (\mu | \nu) ~|~ \mu(0) = \nu, \mu(T) = \xi, t \in [0,T], T \geq 0 \},
  \end{equation}
that is, the infimum cost of traversal from $\nu$ to $\xi$ over all finite time durations. Let us define an equivalence relation on $\MA_1(\ZA)$ as follows. We say $\nu \sim \xi$ if $V(\xi | \nu) = V(\nu | \xi) = 0$. Using a later result (the third part of Lemma \ref{lem:S_T-bounding}), it is easy to show that the set of points that are equivalent to each other is closed and therefore compact (being a closed subset of the compact set $\MA_1(\ZA)$).

We will generalize Theorem \ref{thm:invariant} under the following assumption on the dynamical system corresponding to the McKean-Vlasov equation (\ref{eqn:mckvla-dynamics}):
\begin{itemize}
  \item[({\bf B})] There exist a finite number of compact sets $K_1, K_2, \ldots, K_l$ such that
  \begin{enumerate}
    \item $\nu_1, \nu_2 \in K_i$ implies $\nu_1 \sim \nu_2$.
    \item $\nu_1 \in K_i, \nu_2 \notin K_i$ implies $\nu_1 \nsim \nu_2$.
    \item Every $\omega$-limit set of the McKean-Vlasov equation $\dot{\mu}(t) = A_{\mu(t)}^* \mu(t)$ is contained in one of the $K_i$.
  \end{enumerate}
\end{itemize}

Under the hypothesis of Theorem \ref{thm:invariant}, assumption ({\bf B}) holds with $l=1$, $K_1 = \{ \xi_0 \}$. We saw in a remark following Theorem \ref{thm:invariant} that $s(\xi_0) = 0$. In the general case, we shall see later in Lemma \ref{lem:multiple-equilibria-basevalue} that the rate function $s$ is constant within each $K_i$, and so let $s_i$ be the value of $s(\cdot)$ over $K_i$, for $i = 1, \ldots, l$. In order to specify the values of $s_1, \ldots, s_l$, define
\begin{eqnarray}
  \tilde{V}(K_i, K_j) = \inf_{T > 0} \left\{ S_{[0,T]}(\mu | \mu(0)) ~|~  \mu(0) \in K_i, \mu(T) \in K_j, \right. \nonumber \\
  \label{eqn:V-tilde}
  \left.  \mu(t) \notin \cup_{i' \neq i,j} K_{i'} \mbox{ for } t \in [0,T] \right\}.
\end{eqnarray}
If the set is empty, the infimum is taken to be $+\infty$. Let us also define
\begin{eqnarray}
  \label{eqn:V-KiKj}
  V(K_i, K_j) & = & \left. V(\xi | \nu) \right|_{\nu \in K_i, \xi \in K_j}.
\end{eqnarray}
Again using Lemma \ref{lem:multiple-equilibria-basevalue}, one can show that the above value is independent of $\nu \in K_i$ and $\xi \in K_j$. Also, as indicated by Freidlin and Wentzell (\cite[p.171]{Freidlin}) and by Lemma \ref{lem:FW-1.6} of the Appendix, one can easily argue that
\begin{eqnarray*}
  V(K_i,K_j) & = & \tilde{V}(K_i,K_j) \wedge \min_{i_1} [\tilde{V}(K_i, K_{i_1}) + \tilde{V}(K_{i_1}, K_j)] \\
  & & \wedge \min_{i_1, i_2} [\tilde{V}(K_i, K_{i_1}) + \tilde{V}(K_{i_1}, K_{i_2}) + \tilde{V}(K_{i_2}, K_j)] \\
  & & \wedge \cdots \wedge \min_{i_1, \ldots, i_{l-2}} [\tilde{V}(K_i, K_{i_1}) + \cdots + \tilde{V}(K_{i_{l-2}}, K_j)].
\end{eqnarray*}

Consider the indices $\{1,2,\ldots,l\}$ for the compact sets $K_1, K_2, \ldots, K_l$. Let $\mathbb{G}\{i\}$ be the set of all {\em directed graphs} on the vertex set $\{1,2,\ldots,l\}$ such that
\begin{itemize}
  \item there is no outward edge from $i$;
  \item a vertex $j \neq i$ has exactly one outward edge;
  \item there are no closed cycles in the graph.
\end{itemize}
Define
\begin{equation}
  \label{eqn:W-compactset}
  W(K_{i'}) = \min_{\mathcal{G} \in \mathbb{G}\{i'\}} \sum_{(i,j) \in \mathcal{G}} V(K_i, K_j),
\end{equation}
and finally
\begin{equation}
  \label{eqn:s-values}
  s_{i'} = W(K_{i'}) - \min_{i} W(K_i).
\end{equation}

We now state the generalization of Theorem \ref{thm:invariant} under assumption ({\bf B}).

\begin{thm}
  \label{thm:invariant-multiple-equilibria}
  Assume {\em ({\bf A1})} - {\em ({\bf A3})} and {\em ({\bf B})}. The sequence $(\wp^{(N)}, N \geq 1)$ satisfies the large deviation principle with speed $N$ and good rate function $s$ given by
  \begin{equation}
  \label{eqn:s-2-multiple-equilibria}
    s(\xi) = \inf_{l'} \inf_{\hat{\mu}} \left[ s_{l'} + \int_0^{+\infty} \Big[ \sum_{(i,j) \in \EA} (\hat{\mu}(t)(i)) \lambda_{i,j}(\hat{\mu}(t))
             \tau^* \left( \frac{\hat{l}_{i,j}(t)}{\lambda_{i,j}(\hat{\mu}(t))} - 1 \right) \Big] ~dt \right]
  \end{equation}
where the second infimum is over all $\hat{\mu}$ that are solutions to the dynamical system $\dot{\hat{\mu}}(t) = -\hat{L}(t)^* \hat{\mu}(t)$ for some family of rate matrices $\hat{L}(\cdot)$, with initial condition $\hat{\mu}(0) = \xi$, terminal condition $\hat{\mu}(t) \ra K_{l'}$ as $t \ra +\infty$, and $\hat{\mu}(t) \in \MA_1(\ZA)$ for all $t \geq 0$.
\end{thm}

\section{Large deviations over a finite time duration}
\label{sec:ldp-finite-time}

In the previous section, we stated the main results of the paper on the sequence of invariant measures. We now begin our journey towards the proofs by studying large deviation principles over finite time durations. We first study the empirical measure of paths over finite durations, and then the empirical measure process.

\subsection{Empirical measure}
\label{sec:empirical-measure}

Recall that $x^N \in D([0,T], \ZA^N)$ denotes the full description of all the $N$ particles. Let $G_N$ denote the mapping that takes the full description $x^N$ to the empirical measure
\[
  G_N : (x_n, 1 \leq n \leq N) \in D([0,T], \ZA^N) \mapsto \frac{1}{N} \sum_{n=1}^N \delta_{x_n} \in \MA_1(D([0,T], \ZA)).
\]
Given the random variable $X^N$, the random empirical measure, denoted $M_N$, is thus
\[
  M_N = G_N(X^N) \in \MA_1(D([0,T], \ZA)).
\]
Clearly, the law of $M_N$ depends on the initial condition $z^N$ only through its empirical measure $\nu_N = (1/N) \sum_{n=1}^N \delta_{z_n}$. Write $P^{(N)}_{\nu_N}$ for the law of $M_N$, the push forward of $\mathbb{P}^{(N)}_{z^N}$ under the mapping $G_N$, that is, $P^{(N)}_{\nu_N} = \mathbb{P}^{(N)}_{z^N} \circ G_N^{-1}$.

The $\MA_1^{(N)}(\ZA)$-valued cadlag empirical process is
\[
  \mu_N : t \in [0,T] \mapsto \mu_N(t) = \frac{1}{N} \sum_{n=1}^N \delta_{\{X^{(N)}_n(t)\}} \in \MA_1^{(N)}(\ZA),
\]
and the corresponding mapping is denoted
\[
  \gamma_N : (x_n, 1 \leq n \leq N) \in D([0,T], \ZA^N) \mapsto \mu_N : [0,T] \rightarrow \MA_1^{(N)}(\ZA).
\]
Observe that $\mu_N(0) = \nu_N$, and that $\mu_N(t)$ is the projection $\pi_t(M_N)$ at time $t$, and we write $\mu_N = \pi(M_N)$. We thus have
\[
  \mu_N = \pi(M_N) = \pi(G_N(X^N)) = \gamma_N(X^N).
\]

Consider now a hypothetical tagged particle. When there is no interaction, when all transition rates for $(i,j) \in \EA$ transitions are unity, and when all other transition rates are 0, we can define the evolution of the tagged particle by the law $P_{z}$ which is the unique solution to the martingale problem in $D([0,T], \ZA)$ associated with the generator $A^o$ operating on bounded measurable functions $\Phi$ on $\ZA$ according to
\[
  A^o \Phi(i) = \sum_{j \in \ZA_i} 1 \cdot ( \Phi(j) - \Phi(i) )
\]
and the initial condition $z$. The existence of a solution and the solution's uniqueness hold because the transition rates are upper bounded (see \cite[Problem 4.11.15]{Ethier}). For any fixed $\mu \in D([0,T], \MA_1(\ZA))$, let $P_z(\mu)$ be the unique solution to the martingale problem in $D([0,T], \ZA)$ associated with the (time-varying) generator
\begin{equation}
  \label{eqn:1-particle-generator}
  A_{\mu(t)} \Phi(i) = \sum_{j \in \ZA_i} \lambda_{i,j}(\mu(t)) \cdot ( \Phi(j) - \Phi(i) )
\end{equation}
and the initial condition $z$. This is consistent with (\ref{eqn:rate-matrix}) because when we view $\Phi$ and $A_{\xi} \Phi$ as column vectors, then the vector $A_{\xi} \Phi$ is the result of the rate matrix $A_{\xi} = (\lambda_{i,j}(\xi))_{(i,j \in \ZA \times \ZA)}$ right-multiplied by the vector $\Phi$. Again, by the upper boundedness of $\lambda_{i,j}(\cdot)$ (from assumptions ({\bf A2}-{\bf A3})), $P_z(\mu)$ is unique, and the density of $P_z(\mu)$ with respect to $P_z$ can be written as (see \cite[eqn. (2.4)]{Leonard})
\begin{equation}
  \label{eqn:RND-pathlevel}
  \frac{dP_z(\mu)}{dP_z}(x) = \exp \{ h_1(x; \mu) \}
\end{equation}
where
\begin{eqnarray}
  \label{eqn:h1}
  h_1(x;\mu) =
    \sum_{0 \leq t \leq T} \mathbf{1}_{\{x_t \neq x_{t-}\}} \log \lambda_{x_{t-}, x_t} (\mu(t-)) \\
     - \int_{[0,T]} \sum_{j \in \ZA_{x_t}} \left( \lambda_{x_t, j}(\mu(t)) - 1 \right) ~dt. \nonumber
\end{eqnarray}

Consider the product distribution $\mathbb{P}^{o,(N)}_{z^N} = \bigotimes_{n=1}^N P_{z_n}$ where the $N$ particles evolve independently, with the $n$th particle's initial condition being $z_n$. Again with $\nu_N = (1/N) \sum_{n=1}^N \delta_{z_n}$, let $P^{o,(N)}_{\nu_N} = \mathbb{P}^{o,(N)}_{z^N} \circ G_N^{-1}$. A simple application of Girsanov's formula yields that (see \cite[eqn. (2.8)]{Leonard})
\begin{equation}
  \label{eqn:RND}
  \frac{dP^{(N)}_{\nu_N}}{dP^{o,(N)}_{\nu_N}} (Q) = \exp \{ N h(Q) \}
\end{equation}
where $h$ is related to $h_1(\cdot; \cdot)$ as follows: for a $Q \in \MA_1(D([0,T],\ZA))$,
\begin{equation}
  \label{eqn:h(Q)}
  h(Q) = \int_{D([0,T], \ZA)} h_1(x; \pi(Q)) ~Q(dx).
\end{equation}

Let us define the spaces and topologies of interest. Similar to \cite{Leonard}, consider the Polish space $(\XA, d)$ where
\begin{eqnarray*}
  \lefteqn{ \XA = \big\{ x \in D([0,T], \ZA) \mid \sum_{0 < t \leq T} {\bf 1}_{ \{x(t) \neq x(t-)\} } < +\infty,  } \\
  & & \quad \quad  \mbox{ and for each } t \in (0,T] \mbox{ with } x(t) \neq x(t-), \mbox{ we have } x(t) \in \ZA_{x(t-)} \big\}
\end{eqnarray*}
with metric
\[
  d(x,y) = d_{\textsf{Sko}}(x,y) + |\varphi(x) - \varphi(y)|, \quad x,y \in \XA
\]
where $d_{\textsf{Sko}}$ stands for the Skorohod (complete) metric, and
\[
  \varphi: x \in \XA \mapsto \sum_{0 < t \leq T} {\bf 1}_{ \{x(t) \neq x(t-)\} } \in \mathbb{R}
\]
denotes the number of jumps. $\varphi$ is nonnegative and continuous (see \cite[p. 299]{Leonard}). For a function $f : \XA \ra \mathbb{R}$, define
\[
  || f ||_{\varphi} = \sup_{x \in \XA} \frac{|f(x)|}{1 + \varphi(x)},
\]
denote
\[
  C_{\varphi}(\XA) = \left\{ f \mid f:\XA \ra \mathbb{R} \mbox{ is continuous and } || f ||_{\varphi} < +\infty \right\}
\]
and
\[
  \MA_{1,\varphi}(\XA) = \left\{ Q \in \MA_1(\XA) \mid \int_{\XA} \varphi~ dQ < +\infty \right\}.
\]
This is a subset of the algebraic dual $C_{\varphi}(\XA)^*$ of $C_{\varphi}(\XA)$. Endow the set $\MA_{1,\varphi}(\XA)$ with the weak* topology $\sigma(\MA_{1,\varphi}(\XA), C_{\varphi}(\XA))$, the weakest topology under which $Q_N \ra Q$ as $N \rightarrow +\infty$ if and only if
\[
  \int_{\XA} f~dQ_N \ra \int_{\XA} f~dQ \quad \mbox{ for each } f \in C_{\varphi}(\XA).
\]
This topology is obviously finer than the topology of weak convergence in $\MA_1(\XA)$.

For a measure $\nu \in \MA_1(\ZA)$, let $P$ be the mixture given by
\begin{equation}
  \label{eqn:P}
  dP(x) = \sum_{z \in \ZA} \nu(z) dP_z(x).
\end{equation}
and let $P(\mu)$ be the mixture given by
\begin{equation}
  \label{eqn:P(mu)}
  dP(\mu)(x) = \sum_{z \in \ZA} \nu(z) dP_z(\mu)(x).
\end{equation}

Define the relative entropy $H : \MA_{1, \varphi}(\XA) \rightarrow [0, +\infty]$ of $Q$ (with respect to $P$) as
\begin{equation}
  \label{eqn:relative-entropy}
  H(Q | P) = \left\{
             \begin{array}{cl}
               \int_{\XA} \log \left( \frac{dQ}{dP} \right) ~dQ & \mbox{if } Q \ll P \\
               +\infty & \mbox{otherwise}.
             \end{array}
           \right.
\end{equation}
Also define the function $J : \MA_{1, \varphi}(\XA) \rightarrow [0, +\infty]$
  \begin{equation}
     \label{eqn:J-hat(Q)}
     J(Q) = \sup_{f \in C_{\varphi}(\XA)} \left[ \int_{\XA} f ~dQ - \sum_{z \in \ZA} \nu(z) \log \int_{\XA} e^f ~ dP_z \right].
  \end{equation}

We first state a simple extension of \cite[Th. 2.1]{Leonard}. See remarks following the statement on the nature of the extension.

\begin{thm}
\label{thm:empirical-measure}
Suppose that the initial conditions $\nu_N \rightarrow \nu$ weakly. The sequence $(P^{(N)}_{\nu_N}, N \geq 1)$ satisfies the large deviation principle in the space $\MA_{1,\varphi}(\XA)$, endowed with the weak* topology $\sigma(\MA_{1,\varphi}(\XA), C_{\varphi}(\XA))$, with speed $N$ and good rate function $I(Q) = J(Q) - h(Q)$. Furthermore, for each $Q \in \MA_{1,\varphi}(\XA)$, $I(Q)$ admits the representation
\begin{equation}
  \label{eqn:I(Q)}
  I(Q) = \left\{
            \begin{array}{cl}
              H(Q | P(\pi(Q))) & \mbox{if } Q \circ \pi_0^{-1} = \nu \\
              +\infty    & \mbox{otherwise}.
            \end{array}
    \right.
\end{equation}
\end{thm}

\begin{Remarks}
 This is a mild generalization of \cite[Th.
2.1]{Leonard}. First, as in \cite{Leonard}, the statement is stronger
than usual statements pertaining to the topology of weak convergence
since the topology $\sigma(\MA_{1,\varphi}(\XA), C_{\varphi}(\XA))$ is
finer. Second, while \cite{Leonard} studied $z_n = z_0$ for some fixed
$z_0$ so that $\nu_N = \delta_{z_0}$, we need to consider more general
starting points for each particle, with the only proviso that the
initial empirical measures $\nu_N$ converge weakly to $\nu$. This
generalization also goes beyond the chaotic initial conditions
considered in \cite{DaiPra,Feng-empiricalMeasure}, and
\cite{DelMoral-Zajic}. Third, there is another difference with
\cite{Leonard} in that not all transitions are allowed, but only those
in $\EA$. To get the same results, irreducibility of the graph with
vertices $\ZA$ and directed edges $\EA$ suffices (assumption ({\bf
A1})).
\end{Remarks}

\begin{proof}
This is a straightforward extension of the proof of \cite[Th. 2.1]{Leonard}. The arguments needed for the extension are highlighted in Section \ref{sec:Proof-empirical-measure} for the sake of completeness.
\end{proof}

\subsection{Empirical process}
\label{sec:empirical-flow}

Recall the mapping
\[
  \gamma_N : (x_n, 1 \leq n \leq N) \in D([0,T], \ZA^N) \mapsto \mu_N : [0,T] \rightarrow \MA_1(\ZA).
\]
The flow $\mu_N$ takes values in the space $D([0,T], \MA_1(\ZA))$. Equip this space with the metric
\begin{equation}
  \label{eqn:metric-flow-space}
  \rho_T (\xi, \xi') = \sup_{0 \leq t \leq T} \rho_0(\xi_t, \xi'_t), \quad \xi, \xi' \in D([0,T], \MA_1(\ZA)),
\end{equation}
where $\rho_0$ is taken to be a metric on $\MA_1(\ZA)$ that metrizes the topology of weak convergence on $\MA_1(\ZA)$. The space $D([0,T], \MA_1(\ZA))$ with the metric $\rho_T$ is not separable. We are now interested in the law $p^{(N)}_{\nu_N}$ of $\mu_N$ which is the push forward $p^{(N)}_{\nu_N} = \mathbb{P}^{(N)}_{z^N} \circ \gamma_N^{-1}$. Since $\mu_N = \pi (M_N)$, we can also write $p^{(N)}_{\nu_N}$ as the push forward $p^{(N)}_{\nu_N} = P^{(N)}_{\nu_N} \circ \pi^{-1}$.

\begin{lemma}
  \label{lem:continuity-pi}
  The mapping $\pi : \MA_{1,\varphi}(\XA) \rightarrow D([0,T], \MA_1(\ZA))$ is continuous at each $Q \in \MA_{1,\varphi}(\XA)$ where $J(Q) < +\infty$. In particular, for each $t \in [0,T]$, the projection $\pi_t : \MA_{1,\varphi}(\XA) \rightarrow \MA_1(\ZA)$ is continuous at each $Q \in \MA_{1,\varphi}(\XA)$ where $J(Q) < +\infty$.
\end{lemma}

\begin{proof}
The proof is included immediately after Lemma \ref{lem:continuity-pi-appendix} in Section~\ref{subsec:hQ-continuity}.
\end{proof}

Recall the definition of $\tau$ and its Legendre conjugate $\tau^*$ in (\ref{eqn:tau}) and (\ref{eqn:tau*}), respectively. For $\theta : \ZA \rightarrow \mathbb{R}$, $\xi \in \MA_1(\ZA)$, define
\[
  ||| \theta |||_{\xi} = \sup_{\Phi: \ZA \rightarrow \mathbb{R}} \left\{ \sum_{i \in \ZA} \theta(i) \cdot \Phi(i) - \sum_{(i,j) \in \EA} \tau(\Phi(j) - \Phi(i)) \cdot \xi(i) \cdot \lambda_{i,j}(\xi)  \right\}.
\]
The rate function of interest will be the following. For a $\nu \in \MA_1(Z)$ and $\mu \in D([0,T], \MA_1(Z))$, define
\begin{equation}
  \label{eqn:S(mu-nu)}
  S_{[0,T]}(\mu | \nu) = \left\{
    \begin{array}{ll}
      \int_{[0,T]} ||| \dot{\mu}(t) - A_{\mu(t)}^* \mu(t)  |||_{\mu(t)} ~dt & \mbox{ if } \mu(0) = \nu \mbox{ and } \mu \in \A \\
      +\infty & \mbox{ otherwise},
    \end{array}
  \right.
\end{equation}
where $\A$ is the set of absolutely continuous functions on $[0,T]$. The connection between (\ref{eqn:S(mu-nu)}) and (\ref{eqn:finite-rate-evaluation}) will become clear in the following assertion.

\begin{thm}
  \label{thm:flow-theorem}
  (a) Suppose that the initial conditions $\nu_N \rightarrow \nu$ weakly. Then the sequence $(p^{(N)}_{\nu_N}, N \geq 1)$ satisfies the large deviation principle in the space $D([0,T], \MA_1(\ZA))$ (under the topology induced by the metric $\rho_T$, see (\ref{eqn:metric-flow-space})) with speed $N$ and good rate function $S_{[0,T]}(\mu | \nu)$.

  (b) If the path $\mu \in D([0,T], \MA_1(\ZA))$ has $S_{[0,T]}(\mu | \nu) < +\infty$, then $\mu \in \A$ and there exist rates $(l_{i,j}(t), t \in [0,T], ~(i,j) \in \EA)$ such that
  \begin{itemize}
    \item $\dot{\mu}(t) = L(t)^* \mu(t)$ where $L(t)$ is the rate matrix associated with the time-varying rates $(l_{i,j}(t), (i,j) \in \EA)$ and $L(t)^*$ is its adjoint;
    \item the good rate function $S_{[0,T]}(\mu | \nu)$ is given by the equation (\ref{eqn:finite-rate-evaluation}).
  \end{itemize}

  (c) Suppose that the following hold: $\mu \in \A$, $\mu(0) = \nu$, there exist time-varying rates $(l_{i,j}(t), t \in [0,T], ~(i,j) \in \EA)$ such that the associated rate matrix $L(t)$ satisfies $\dot{\mu}(t) = L(t)^* \mu(t)$, and the right-hand side of (\ref{eqn:finite-rate-evaluation}) is finite. Then the good rate function $S_{[0,T]}(\mu | \nu)$ evaluated at $\mu$ is given by~(\ref{eqn:finite-rate-evaluation}).
\end{thm}

\begin{proof}
With the generalization available in Theorem \ref{thm:empirical-measure}, the proof (of all three statements) is identical to the proof of \cite[Th. 3.1]{Leonard}. See also \cite[Th. 2]{DjehicheKaj} for statements on the nature of the rate function.
\end{proof}

\begin{Remarks}
 1. Yet again, this is a mild generalization of \cite[Th.
3.1]{Leonard}, and of the results in \cite{DaiPra,Feng-empiricalFlow}, and \cite{DelMoral-Zajic} that assume
chaotic initial conditions. We allow any arbitrary sequence of initial
conditions $\nu_N$ so long as $\nu_N \rightarrow \nu$ weakly. See a
consequence in Corollary \ref{cor:uniform-ldp-flows} below.

2. Observe from (\ref{eqn:finite-rate-evaluation}) that if the rate function $S_{[0,T]}(\mu | \nu) = 0$ then $\mu$ must be the unique solution to the McKean-Vlasov equation $\dot{\mu}(t) = A_{\mu(t)}^* \mu(t)$ with initial condition $\mu(0) = \nu$. (That the solution is unique follows from Lipschitz assumption ({\bf A2}) which implies the well-posedness of the ODE (\ref{eqn:mckvla-dynamics})). The claim $p^{(N)}_{\nu_N} \rightarrow \delta_{\mu(\cdot)}$ follows.

3. When a path $\mu$ is such that $S_{[0,T]}(\mu | \nu) < +\infty$, the first bullet in the second statement of Theorem \ref{thm:flow-theorem} says that there is a control (tilt), given by the rate matrix $L(t)$, such that the normal limiting trajectory under this control is $\mu$. The expression (\ref{eqn:S(mu-nu)}) and the alternate expression (\ref{eqn:finite-rate-evaluation}) reinforce the notion that $S_{[0,T]}(\mu | \nu)$  is the {\em cost} of transporting the system along the trajectory $\mu$.

4. For any finite $B \geq 0$ and any $\nu \in \MA_1(\ZA)$, the set $\{ \mu \in \A ~|~ S_{[0,T]}(\mu | \nu) \leq B \}$ is compact since $S_{[0,T]}(\cdot | \nu)$ is a good rate function.
\end{Remarks}

The following is a straightforward corollary which shows that $S_{[0,T]}(\mu | \nu)$ is a rate function for a large deviation principle that holds uniformly in the initial point.

\begin{corollary}
\label{cor:uniform-ldp-flows}
For any compact set $K \subset \MA_1(\ZA)$, any closed set \\
$F \subset D([0,T], \MA_1(\ZA))$, and any open set $G \subset D([0,T], \MA_1(\ZA))$, we have
\begin{eqnarray}
  \label{eqn:uniform-ldp-ub}
  \limsup_{N \ra \infty} \frac{1}{N} \log \sup_{\nu \in K} p^{(N)}_{\nu} \left\{ \mu_N \in F \right\} & \leq & - \inf_{\nu \in K, \mu \in F} S_{[0,T]}(\mu | \nu), \\
  \label{eqn:uniform-ldp-lb}
  \liminf_{N \ra \infty} \frac{1}{N} \log \inf_{\nu \in K} p^{(N)}_{\nu} \left\{ \mu_N \in G \right\} & \geq & - \sup_{\nu \in K} \inf_{\mu \in G} S_{[0,T]}(\mu | \nu).
\end{eqnarray}
\end{corollary}

\begin{proof}
This is immediate from \cite[Cor. 5.6.15]{DZ} in conjunction with Theorem \ref{thm:flow-theorem}.
\end{proof}

\subsection{Empirical measure at initial and terminal times}
\label{sec:empirical-flow-initial-terminal-time}

We first study the behavior of the empirical measure at terminal time alone. For $T \geq 0$, recall that the law of the empirical process $\mu_N$ with initial condition $\mu_N(0) = \nu_N$ (arising from $z^N$) is $p^{(N)}_{\nu_N} = \mathbb{P}^{(N)}_{z^N} \circ \pi^{-1}$. The empirical measure at terminal time $T$ is $\mu_N(T)$; its law is the push forward $p^{(N)}_{\nu_N, T} = \mathbb{P}^{(N)}_{z^N} \circ \pi_T^{-1}$. The following result is an easy consequence of the contraction principle \cite[Th. 4.2.1]{DZ}.

\begin{thm}
  Suppose that the initial conditions $\nu_N \rightarrow \nu$ weakly. Then the family $(p^{(N)}_{\nu_N,T}, N \geq 1)$ satisfies the large deviation principle in $\MA_1(\ZA)$ with speed $N$ and good rate function
  \begin{equation}
    \label{eqn:S(xi-nu)}
    S_T(\xi | \nu) = \inf \{ S_{[0,T]}(\mu | \nu) \mid \mu(0) = \nu, \mu(T) = \xi, \mu \in \A \}.
  \end{equation}
  Furthermore $S_T(\xi | \nu)$ is bounded for all $\nu, \xi \in \MA_1(\ZA)$, and we may restrict attention in the infimum to $\mu \in \A$ that also satisfy $S_{[0,T]}(\mu | \nu) < +\infty$. Moreover, the infimum is attained, and there exist $\overline{\mu} \in \A$ with $\overline{\mu}(0) = \nu$, $\overline{\mu}(T) = \xi$, and rates $(l_{i,j}(t), t \in [0,T], ~(i,j) \in \EA)$ with associated rate matrix $L(t)$ such that $\dot{\overline{\mu}}(t) = L(t)^* \overline{\mu}(t)$, $S_T(\xi | \nu) = S_{[0,T]}(\overline{\mu} | \nu)$, and $S_{[0,T]}(\overline{\mu} | \nu)$ satisfies (\ref{eqn:finite-rate-evaluation}).
\end{thm}

\begin{proof}
Recall that the metric on $D([0,T], \MA_1(\ZA))$ is $\rho_T$ given in (\ref{eqn:metric-flow-space}). It follows that the mapping $\mu \in D([0,T], \MA_1(\ZA)) \mapsto \mu(T) \in \MA_1(\ZA)$ is continuous. The statement on the validity of the large deviation principle with good rate function (\ref{eqn:S(xi-nu)}) follows by the contraction principle \cite[Th. 4.2.1]{DZ}. In the second statement of the next lemma, we show that $S_T(\xi | \nu)$ is bounded for all $\nu, \xi \in \MA_1(\ZA)$. The goodness of the rate function $S_{[0,T]}(\mu | \nu)$ implies that the infimum in (\ref{eqn:S(xi-nu)}) is attained. The rest follow from Theorem~\ref{thm:flow-theorem}.
\end{proof}

Let us now prove the boundedness and related properties of $S_T(\xi | \nu)$.

\begin{lemma}
  \label{lem:S_T-bounding}
  The following statements hold.
  \begin{itemize}
    \item There exists a constant $C_1(T) < +\infty$ such that for any $\xi, \nu \in \MA_1(\ZA)$, there is a piecewise linear and continuous path $\mu$ having constant velocity in each linear segment and $S_{[0,T]}(\mu | \nu) \leq C_1(T)$.
    \item For any $\xi, \nu \in \MA_1(\ZA)$, we have $S_T(\xi | \nu) \leq C_1(T)$.
    \item There exists a constant $C_2 < +\infty$ such that for every $\varepsilon > 0$, there is a $\delta \in (0, \varepsilon)$ such that $\rho_0(\nu, \xi) < \delta$ implies $S_{\varepsilon}(\xi | \nu) \leq C_2 \varepsilon$.
  \end{itemize}
\end{lemma}

\begin{proof}
The main idea is to show that the difficulty of passage near the neighborhood of a constant velocity straight line path is bounded. See Section \ref{sec:ProofOfS_T-bounding}.
\end{proof}

We next have a useful uniform continuity result.

\begin{lemma}
  \label{lem:uniform-continuity}
  The mapping $(\nu, \xi) \mapsto S_T(\xi | \nu)$ is uniformly continuous.
\end{lemma}

\begin{proof}
See Section \ref{sec:uniform-continuity}.
\end{proof}

Thus far, the only condition we imposed on the initial conditions were that $\nu_N \rightarrow \nu$ weakly. Let $\wp^{(N)}_0$ denote the law of the initial empirical measure $\mu_N(0)$ and let $\wp^{(N)}_{0,T}$ denote the joint law of $(\mu_N(0), \mu_N(T))$. We now consider $(\wp^{(N)}_{0,T}, N \geq 1)$.

\begin{thm}
  \label{thm:JointLDP}
  Suppose that the sequence $(\wp^{(N)}_0, N \geq 1)$ satisfies the large deviation principle with speed $N$ and good rate function $s : \MA_1(Z) \rightarrow [0, +\infty]$. Then the sequence of joint laws $(\wp^{(N)}_{0,T}, N \geq 1)$ satisfies the large deviation principle with speed $N$ and good rate function
  \[
    S_{0,T}(\nu, \xi) = s(\nu) + S_T(\xi | \nu).
  \]
\end{thm}

\begin{proof}
See Section \ref{sec:ProofJointLDP}.
\end{proof}

Before we close this section, we state a useful result on the uniform continuity of the quasipotential in (\ref{eqn:quasipotential}). This is analogous to Lemma \ref{lem:uniform-continuity}, except that the time durations are finite but otherwise arbitrary.

\begin{lemma}
\label{lem:V-uniform-continuity}
The mapping $(\nu, \xi) \mapsto V(\xi | \nu)$ is uniformly continuous.
\end{lemma}

\begin{proof}
See the last part of Section \ref{sec:uniform-continuity}.
\end{proof}

\section{Invariant measure: A control theoretic approach}
\label{sec:invariant-measure}

Recall that by assumption ({\bf A1}), for each $N \geq 1$, the finite-state continuous-time Markov chain $\mu_N$ is irreducible, and hence has a unique invariant measure, which we denoted $\wp^{(N)}$. In this section, we establish the large deviation principle for $(\wp^{(N)}, N \geq 1)$ as stated in Theorems \ref{thm:invariant} and \ref{thm:invariant-multiple-equilibria}.

The outline of our control theoretic approach is the following.
\begin{itemize}
  \item In Lemma \ref{lem:HJB}, we first establish a subsequential large deviation principle. We shall also establish, via the contraction principle, that the rate function $s$ satisfies a dynamic programming equation (see (\ref{eqn:HJB})). This equation naturally suggests a control problem with an associated running cost.

 \item There will be multiple solutions to (\ref{eqn:HJB}). But the
     rate function that we are after will satisfy a further\vadjust{\eject}
     condition. In Lemma \ref{lem:one-trajectory}, we shall show
     the existence of one single optimal path of infinite duration
     and shall extract a recursive equation for $s$ from
     (\ref{eqn:HJB}), a further condition that the rate function
     must satisfy. The heart of the control-theoretic approach lies
     in this step.

  \item We then argue in Lemma \ref{lem:converge-w-limit} that this optimal path must end up within a set that is positively invariant to the time-reversed McKean-Vlasov dynamics.

  \item The above steps fix the rate function at all points outside the $\omega$-limit sets, assuming the values at the $\omega$-limit sets.

  \item Subsection \ref{subsec:Inv-measure-ugase} then argues that in case of a unique globally asymptotically stable equilibrium, the rate function is zero at the equilibrium. Subsection \ref{subsec:Inv-measure-multiple-eq} falls back on the approach of Freidlin and Wentzell \cite[Ch. 6]{Freidlin} under assumption ({\bf B}) to fix the values at the $\omega$-limit sets. This fixes the rate function uniquely for all subsequential large deviation principles, and the main results follow.
\end{itemize}

We begin by establishing subsequential large deviation principles and the dynamic programming equation.

\begin{lemma}
\label{lem:HJB}
For any sequence of natural numbers going to $+\infty$, there exists a subsequence $(N_k, ~k \geq 1)$ such that $(\wp^{(N_k)}, ~k \geq 1)$ satisfies the large deviation principle with speed $N_k$ and a good rate function $s$ that satisfies
\begin{equation}
  \label{eqn:HJB}
  s(\xi) = \inf_{\nu \in \MA_1(\ZA)} [s(\nu) + S_{T}(\xi | \nu)] \quad \mbox{for every } T > 0.
\end{equation}
Furthermore, $s \geq 0$ and there exists a $\nu^* \in \MA_1(\ZA)$ such that $s(\nu^*) = 0$.
\end{lemma}

\begin{proof}
Since the topology on $\MA_1(\ZA)$ with metric $\rho_0$ has a countable base, and because  $\MA_1(\ZA)$ is compact, by \cite[Lem. 4.1.23]{DZ}, there is a subsequence ${N_k \ra +\infty}$ of the given sequence such that $(\wp^{(N_k)}, N_k \geq 1)$ satisfies the large deviation principle with speed $N_k$ and a good rate function $s : \MA_1(\ZA) \ra [0, +\infty]$.

We now verify (\ref{eqn:HJB}). Fix an arbitrary $T > 0$. By Theorem \ref{thm:JointLDP}, with $\wp^{(N)}_0 = \wp^{(N)}$, the invariant measure, the sequence of joint laws $(\wp^{(N_k)}_{0,T}, ~k \geq 1)$ satisfies the large deviation principle along the subsequence $(N_k, ~k \geq 1)$ with speed $N_k$ and good rate function $S_{0,T}(\nu, \xi) = s(\nu) + S_T(\xi|\nu)$. By the contraction principle, the sequence of terminal laws $(\wp^{(N_k)}_T, ~k \geq 1)$ satisfies the large deviation principle along the subsequence with the good rate function
\begin{equation}
  \label{eqn:HJB-precursor}
  \inf_{\nu \in \MA_1(\ZA)} S_{0,T}(\nu, \xi) = \inf_{\nu \in \MA_1(\ZA)} [s(\nu) + S_T(\xi | \nu)].
\end{equation}
But $\wp^{(N)}$ is invariant to time shifts which yields $\wp^{(N)}_T = \wp^{(N)}_0 = \wp^{(N)}$. The infimum on the left-hand side of (\ref{eqn:HJB-precursor}) must therefore evaluate to $s(\xi)$, which yields (\ref{eqn:HJB}).

Rate functions are nonnegative and have infimum value of 0, that is, $s \geq 0$ and $\inf_{\nu \in \MA_1(\ZA)} s(\nu) = 0$. Since $s$ is a good rate function, the infimum 0 is attained at some point; call it $\nu^*$. The proof is now complete.
\end{proof}

As indicated earlier, there are multiple solutions to (\ref{eqn:HJB}). Indeed, $s(\cdot) \equiv 0$ is one of them. In order to identify a further condition that $s$ must satisfy, we now identify a control problem associated to (\ref{eqn:HJB}). To do this, we shall now consider paths that are of time-duration $mT$ that end at $\xi$. Since the terminal condition is fixed, it would be convenient to fix the terminal time as 0 and look at negative times; in particular, paths in the time interval $[-mT, 0]$ for $m \geq 1$. But then, we may reverse time and consider the dynamical system
\[
  \dot{\hat{\mu}}(t) = - \hat{L}(t)^* \hat{\mu}(t), \quad \hat{\mu}(0) = \xi, \quad t \in [0, mT], \quad m \geq 1.
\]
We use hats as in $\hat{\mu}, \hat{l}_{i,j}, \hat{\lambda}_{i,j}, \hat{L}$ to denote quantities where time flows in the opposite direction with reference to the direction under the McKean-Vlasov dynamics. In particular,
\begin{eqnarray*}
  \hat{\mu}(t) & = & \mu(mT-t) \\
  \hat{l}_{i,j}(t) & = & l_{i,j}(mT-t), \mbox{for all } i,j \in \ZA \\
  \hat{L}(t) & = & L(mT - t).
\end{eqnarray*}
for $t \in [0,mT]$. Also, for uniformity in notation, let $\hat{\lambda}_{i,j}(\cdot) = \lambda_{i,j}(\cdot)$ for all $(i,j)$ pairs. One then views the $\hat{L}(t)$ above as the control at time $t$ when the state is $\hat{\mu}(t)$ with cost function at time $t$ given by
\[
  \hat{r}(\hat{\mu}(t), \hat{L}(t)) = \sum_{(i,j) \in \EA} (\hat{\mu}(t)(i)) \hat{\lambda}_{i,j}(\hat{\mu}(t)) \tau^* \left( \frac{\hat{l}_{i,j}(t)}{\hat{\lambda}_{i,j}(\hat{\mu}(t))} - 1 \right).
\]
Observe that the cost function is zero when $\hat{l}_{i,j}(t) = \hat{\lambda}_{i,j}(\hat{\mu}(t))$ for almost every $t$ in the time duration of interest. The total cost is
\begin{equation}
  \label{eqn:total-cost}
  \int_{[0,mT]} \Big[ \sum_{(i,j) \in \EA} (\hat{\mu}(t)(i)) \hat{\lambda}_{i,j}(\hat{\mu}(t)) \tau^* \left( \frac{\hat{l}_{i,j}(t)}{\hat{\lambda}_{i,j}(\hat{\mu}(t))} - 1 \right) \Big] ~dt.
\end{equation}
which is simply $S_{[0,mT]}(\mu|\mu(0))$ of (\ref{eqn:finite-rate-evaluation}) as can be verified by a change of variable in (\ref{eqn:total-cost}) that takes $t$ to $mT-t$.

The following lemma establishes the existence of one optimal path $\hat{\mu}$ of infinite duration, starting at $\xi$.

\begin{lemma}
  \label{lem:one-trajectory}
  For each $\xi \in \MA_1(\ZA)$, there exists a path $\hat{\mu}:[0,+\infty) \ra \MA_1(\ZA)$ and a family of rate matrices $(\hat{L}(t), t \in [0, +\infty))$ such that $\hat{\mu} (\cdot)$ satisfies the ODE
  \begin{equation}
    \label{eqn:backward-mckvla-dynamics}
    \dot{\hat{\mu}}(t) = -\hat{L}(t)^* \hat{\mu}(t), \quad t \in [0, +\infty)
  \end{equation}
  with initial condition $\hat{\mu}(0) = \xi$, and
  \begin{equation}
    \label{eqn:equality-in-parts}
    s(\xi) = s(\hat{\mu}(mT)) + \int_{[0,mT]} \hat{r}(\hat{\mu}(t), \hat{L}(t)) ~dt \quad \mbox{ for all } m \geq 1.
  \end{equation}
\end{lemma}

\begin{Remark}
While the McKean-Vlasov dynamics or the more general $\dot{\mu}(t) = L(t)^* \mu(t)$ with $L(t)$ being a rate matrix ensures that $\mu(t)$ lies within $\MA_1(\ZA)$, this is not the case for the dynamics given by (\ref{eqn:backward-mckvla-dynamics}). Indeed, at the boundary of $\MA_1(\ZA)$, viewed as a subset of $\mathbb{R}^r$, the velocity for the dynamics in (\ref{eqn:backward-mckvla-dynamics}) points towards a direction of immediate exit from $\MA_1(\ZA)$. The state space for the dynamics of (\ref{eqn:backward-mckvla-dynamics}) is therefore not restricted to $\MA_1(\ZA)$. However, the lemma assures us that the selected path $\hat{\mu}$ stays within the compact subset $\MA_1(\ZA)$ for all time.
\end{Remark}

\begin{proof}
Our approach to prove this is the following. We shall define a topology on a suitable subspace of paths of infinite duration, and then show that we can restrict attention to a compact subset. We shall then argue that there exists a nested sequence of decreasing compact subsets, each of which is nonempty and all of whose elements satisfy the desired properties. The intersection will then be nonempty to yield the desired path.

{\em Step 1}: We shall now restrict attention to paths that lie inside $\MA_1(\ZA)$. For a path $\hat{\mu}: [0,+\infty) \ra \MA_1(\ZA)$, define its restrictions to $[0, mT]$ by
\[
  \hat{\mu} \mapsto \psi_{m}\hat{\mu}(\cdot) = \hat{\mu}^{(m)}(\cdot) : [0,mT] \ra \MA_1(\ZA)
\]
which is the restriction of the path $\hat{\mu}$ to $[0,mT]$. Consider the space of paths of infinite duration with metric
\[
  \rho_{\infty}(\hat{\mu}, \hat{\eta}) = \sum_{m = 1}^{\infty} 2^{-m} \left( \rho_{mT}(\psi_m \hat{\mu}, \psi_m \hat{\eta}) \wedge 1 \right).
\]
Obviously, $\psi_m$ is continuous for each $m$.

We shall also consider reversed restrictions (denoted without hats) defined by
\begin{equation}
  \label{eqn:forward-segments}
  \mu^{(m)}(t) = \hat{\mu}^{(m)}(mT-t) = \hat{\mu}(mT - t), \quad t \in [0,mT].
\end{equation}
Fix a $B \in [0, +\infty)$, and consider the set
\begin{equation}
  \label{eqn:Gamma-infinity}
  \Gamma_{\infty} = \left\{ \hat{\mu}(\cdot) : [0, +\infty) \ra \MA_1(\ZA) \mid \sup_{m \geq 1} S_{[0,mT]}(\mu^{(m)} | \mu^{(m)}(0)) \leq B \right\}.
\end{equation}
Also, with $\eta(t) = \hat{\eta}(mT-t)$ for $t \in [0,mT]$, define
\[
  \Gamma_m = \left\{ \hat{\eta} : [0, mT] \ra \MA_1(\ZA) \mid S_{[0,mT]}(\eta | \eta(0)) \leq B \right\}.
\]

$\Gamma_{\infty}$ is compact. To see this, take an arbitrary infinite sequence $( \hat{\mu}_n, n \geq 1 ) \subset \Gamma_{\infty}$. Since $S_{[0,mT]}$ is a good rate function, using Lemma \ref{lem:speed-up} and Lemma \ref{lem:tau-star-linear-bound}, it is easy to see that each $\Gamma_m$ is compact, and so one can find an infinite subset $\mathbb{V}_1 \subset \mathbb{N}$ such that $( \psi_1 \hat{\mu}_n, n \in \mathbb{V}_1 ) \subset \Gamma_{1}$ converges. Take a further subsequence represented by the infinite subset $\mathbb{V}_2 \subset \mathbb{V}_1$ such that $( \psi_2 \hat{\mu}_n, n \in \mathbb{V}_2 ) \subset \Gamma_{2}$ converges. Continue this procedure and take the subsequence along the diagonal. This subsequence converges for every interval $[0,mT]$. For each $t$ define $\hat{\mu}(t)$ to be the point-wise limit. Since for each $m$, we have $\psi_m \hat{\mu} \in \Gamma_m$, it follows that with $\mu^{(m)}$ defined as in (\ref{eqn:forward-segments}), $\sup_{m \geq 1} S_{[0,mT]}(\mu^{(m)} | \mu^{(m)}(0)) \leq B$, and so $\hat{\mu} \in \Gamma_{\infty}$. Thus $\Gamma_{\infty}$ is sequentially compact, and by virtue of its being a subset of a metric space, $\Gamma_{\infty}$ is compact.

Fix $\xi \in \MA_1(\ZA)$. Consider the time duration $[0,T]$. We know from Lemma \ref{lem:HJB} that there is a $\nu^*$ with $s(\nu^*) = 0$. Using this in (\ref{eqn:HJB}) of Lemma \ref{lem:HJB}, we get
\[
  s(\xi) \leq S_T(\xi | \nu^*) \leq C_1(T)
\]
where the last inequality is due to the second statement in Lemma \ref{lem:S_T-bounding}. Take the constant $B = C_1(T)$; the corresponding $\Gamma_{\infty}$ is compact.

{\em Step 2}: Observe that (\ref{eqn:HJB}) can be viewed as a minimization over path space, with paths $\mu$ of duration $[0,mT]$ ending at $\xi$. Starting from any initial location $\nu$, the minimum value is upper bounded by $B$. Indeed, traverse the McKean-Vlasov path with initial condition $\nu$ for duration $(m-1)T$. This contributes zero to the cost. Then proceed to $\xi$ in $T$ units of time. This costs at most $B = C_1(T)$. The minimum cost to go from $\nu$ to $\xi$ in time $[0,mT]$ is thus at most $B$. If we consider reversed and translated time so that initial time is 0, the reversed paths $\hat{\mu}$ begin at $\xi$ at time 0, have cost at most $B$, and stay in $\MA_1(\ZA)$ for the duration $[0,mT]$.

Let
\begin{eqnarray*}
  \Gamma^*_m = \bigcup_{\nu \in \MA_1(\ZA)} \Big\{ \hat{\mu} ~|~ \mu(t) = \hat{\mu}(mT - t), \mu(0) = \nu, \mu(mT) = \xi, \\
     S_{[0,mT]}(\mu | \nu) = S_{mT}(\xi | \nu) \Big\},
\end{eqnarray*}
that is, the collection of all minimum cost paths $\hat{\mu}$ in $[0,mT]$ from $\xi$ to every location in $\MA_1(\ZA)$; the minimum cost is at most $B$. Clearly, $\Gamma^*_m \subset \Gamma_m$, and $\Gamma^*_m \neq \emptyset, m \geq 1$.

$\Gamma^*_m$ is also compact. This set being a subset of the compact set $\Gamma_m$, it suffices to show that $\Gamma^*_m$ is closed. Let $\hat{\mu}$ be a point of closure of $\Gamma^*_m$. We can then find a sequence $(\hat{\mu}^{(k)}, k \geq 1) \subset \Gamma^*_m$ such that $\lim_{k \ra +\infty} \hat{\mu}^{(k)} = \hat{\mu}$. Clearly, we must have $\hat{\mu}(0) = \xi$. Let $\nu = \hat{\mu}(mT)$. By a simple application of lower semicontinuity of $S_{[0,mT]}(\cdot | \nu)$, Lemma \ref{lem:speed-up}, and Lemma \ref{lem:tau-star-linear-bound}, we must have
\begin{eqnarray*}
  S_{[0,mT]}(\mu | \nu) & \leq & \liminf_{k \ra +\infty} S_{[0,mT]}(\mu^{(k)} | \mu^{(k)}(0)) \\
  & = & \liminf_{k \ra +\infty} S_{mT}(\xi | \mu^{(k)}(0)) \\
  & = & S_{mT}(\xi | \nu)
\end{eqnarray*}
where the last inequality follows because of the continuity of $S_{mT}$ in its arguments. But $S_{mT}(\xi | \nu)$ is the least cost for paths that traverse from $\nu$ to $\xi$ in duration $[0,mT]$. So we must have $S_{[0,mT]}(\mu | \nu) = S_{mT}(\xi | \nu)$, which establishes that $\mu \in \Gamma^*_m$; $\Gamma^*_m$ is therefore closed.

{\em Step 3}: Let us now finish the proof of Lemma \ref{lem:one-trajectory}. Since $\Gamma_m^*$ is nonempty and compact, the continuity of $\psi_m$ implies that $\psi_m^{-1} \Gamma_m^*$ is nonempty and closed. Further, being a closed subset of the compact set $\Gamma_{\infty}$, $\psi_m^{-1} \Gamma_m^*$ is itself compact. A simple dynamic programming argument further shows that $\psi_m^{-1} \Gamma_m^*$ is a nested decreasing sequence of subsets. Their intersection is nonempty. Take a $\hat{\mu}$ in the intersection.

Focusing on the duration $[mT, mT+T]$, the path $t \in [0,T] \mapsto \eta^{(m)}(t) = \hat{\mu}(mT + T - t)$ has $S_{[0,T]}(\eta^{(m)} | \eta^{(m)}(0)) \leq B < +\infty$, and so, by the last part of Theorem \ref{thm:flow-theorem}, we can find rates $L^{(m)}(t)$ such that $\eta^{(m)}$ satisfies $\dot{\eta}^{(m)}(t) = (L^{(m)}(t))^* \eta^{(m)}(t)$, with initial condition $\eta^{(m)}(0)$. Put these pieces of duration $T$ together by defining $\hat{L}(mT+t) = L^{(m)}(T-t)$ for $t \in [0,T]$, and we get a rate matrix $\hat{L}(\cdot)$ defined on $[0, +\infty)$ such that $\hat{\mu}$ is the solution to (\ref{eqn:backward-mckvla-dynamics}) with initial condition $\hat{\mu}(0) = \xi$. The last equality (\ref{eqn:equality-in-parts}) follows because $\hat{\mu}$ attains the minimum in (\ref{eqn:HJB}) for each duration $[0,mT]$.
\end{proof}

The next lemma says that the optimal path above must end up in a specific invariant set for the dynamics given by the time-reversed McKean-Vlasov equation.

\begin{lemma}
\label{lem:converge-w-limit}
Consider the trajectory $\hat{\mu}$ given by Lemma \ref{lem:one-trajectory}. Its $\omega$-limit set, which is the set of its limit points as $t \ra +\infty$ given by
\[
  \Omega = \bigcap_{t > 0} \overline{\{ \hat{\mu}(t'), t' \geq t \}},
\]
is positively invariant for the ODE
\begin{equation}
  \label{eqn:backward-mckvla-dynamics-actual}
  \dot{\hat{\mu}}(t) = - A_{\hat{\mu}(t)}^* \hat{\mu}(t), \quad t \geq 0.
\end{equation}
\end{lemma}

\begin{proof}
Take the $\hat{\mu}$ given by Lemma \ref{lem:one-trajectory}. It remains within $\MA_1(\ZA)$, and satisfies the dynamics $\dot{\hat{\mu}}(t) = - \hat{L}(t)^* \hat{\mu}(t)$ for $t \in [0, +\infty)$ with initial condition $\hat{\mu}(0) = \xi$. Furthermore, (\ref{eqn:equality-in-parts}) holds (for all $m \geq 1$). Since the integrand in (\ref{eqn:equality-in-parts}) is nonnegative, $s(\hat{\mu}(mT))$ must decrease as $m$ increases. But $s$ is bounded between $[0, C_1(T)]$, and so there is an $s^*$ such that $s(\hat{\mu}(mT)) \downarrow s^*$ as $m \uparrow +\infty$.

Consider any arbitrary subsequence of $(\hat{\mu}(mT), m \geq 1)$ and take a subsequential limit $\xi'$. On this subsequence, take a further subsequential limit of $(\hat{\mu}(mT + T), m \geq 1)$ and call it $\nu$. Call the subsequence $(m_k, k \geq 1)$, and consider the paths
\[
  \mu^{(m_k)}(t) = \hat{\mu}(m_k T + T - t), \quad t \in [0,T]
\]
which are of duration $T$ and time reversals of fragments of $\hat{\mu}$. We thus have the subsequential convergence
\begin{equation}
  \label{eqn:subseq-convergence}
  (\mu^{(m_k)}(0), \mu^{m_k}(T)) \ra (\nu, \xi') \mbox{ as } k \ra +\infty.
\end{equation}
Taking limits as $k \ra +\infty$ in (\ref{eqn:equality-in-parts}) and using the fact that $s(\mu^{(m_k)}(0))$ as well as $s(\mu^{(m_k)}(T))$ converge to $s^*$ as $k \ra +\infty$, the integral term, which is easily seen to be $S_{[0,T]}(\mu^{(m)} | \mu^{(m)}(0))$, satisfies
\[
  \limsup_{k \ra +\infty} S_{[0,T]}(\mu^{(m_k)} | \mu^{(m_k)}(0)) = 0.
\]
This fact and the nonnegativity of $S_T$ imply
\[
  \lim_{k \ra +\infty} S_T(\mu^{(m_k)}(T) | \mu^{(m_k)}(0)) = 0.
\]
By the uniform continuity of $S_T$ in both its arguments (Lemma \ref{lem:uniform-continuity}) and by (\ref{eqn:subseq-convergence}), we deduce that $S_T(\xi'|\nu) = 0$. But then the path that goes from $\nu$ to $\xi'$ in time $[0,T]$ and attains $S_T(\xi'|\nu) = 0$ is the McKean-Vlasov path which is the solution to the dynamics
\[
  \dot{\mu}(t) = A_{\mu(t)}^* \mu(t), ~t \in [0,T]
\]
with rate matrix $A_{\mu(t)}$, initial condition $\mu(0) = \nu$, and final condition $\mu(T) = \xi'$. But this implies that $\overline{\mu}(t) = \mu(T-t)$ satisfies (\ref{eqn:backward-mckvla-dynamics}) with $\hat{L}(t) = A_{\overline{\mu}(t)}$ and initial condition $\overline{\mu}(0) = \xi'$. It follows that $\Omega$, the $\omega$-limit set for $\hat{\mu}$, is contained in the $\omega$-limit set for the dynamics (\ref{eqn:backward-mckvla-dynamics-actual}) with initial condition $\hat{\mu}(0) = \xi$. This concludes the proof.
\end{proof}

\subsection{Invariant measure: Unique globally asymptotically stable equilibrium}
\label{subsec:Inv-measure-ugase}

In this subsection, we consider the case when there is a unique globally asymptotically stable equilibrium $\xi_0$.

\begin{lemma}
\label{lem:ugase-basevalue}
If the McKean-Vlasov equation $\dot{\mu}(t) = A_{\mu(t)}^* \mu(t)$ has a unique globally asymptotically stable equilibrium $\xi_0$, then $s(\xi_0) = 0$.
\end{lemma}

\begin{proof}
Consider the McKean-Vlasov dynamics
\[
  \dot{\mu}(t) = A_{\mu(t)}^* \mu(t)
\]
with initial condition $\mu(0) = \nu^*$. By our assumption that $\xi_0$ is the unique globally asymptotically stable equilibrium, $\mu(T) \ra \xi_0$ as $T \ra +\infty$. By the second remark following Theorem \ref{thm:flow-theorem}, the McKean-Vlasov path has zero cost and so $S_T(\mu(T) | \nu^*) = 0$ for each $T > 0$. By (\ref{eqn:HJB}), we then have
\[
  s(\mu(T)) \leq s(\nu^*) + S_T(\mu(T) | \nu^*) = s(\nu^*).
\]
Take limits as $T \ra +\infty$ and use the lower semicontinuity of $s$ to get
\[
  0 \leq s(\xi_0) \leq \liminf_{T \ra +\infty} s(\mu(T)) \leq s(\nu^*) = 0,
\]
whence $s(\xi_0) = 0$.
\end{proof}

The following lemma shows that the rate function for any subsequential large deviation principle is unique.

\begin{lemma}
\label{lem:uniqueness}
If the McKean-Vlasov equation $\dot{\mu}(t) = A_{\mu(t)}^* \mu(t)$ has a unique globally asymptotically stable equilibrium $\xi_0$, then the solution to (\ref{eqn:HJB}) and (\ref{eqn:equality-in-parts}) is unique and is given by
  \begin{equation}
  \label{eqn:s}
    s(\xi) = \inf_{\hat{\mu}} \int_{[0, +\infty)} \Big[ \sum_{(i,j) \in \EA} (\hat{\mu}(t)(i)) \hat{\lambda}_{i,j}(\hat{\mu}(t))
             \tau^* \left( \frac{\hat{l}_{i,j}(t)}{\hat{\lambda}_{i,j}(\hat{\mu}(t))} - 1 \right) \Big] ~dt
  \end{equation}
where the infimum is over all $\hat{\mu}$ that are solutions to the dynamical system $\dot{\hat{\mu}}(t) = -\hat{L}(t)^* \hat{\mu}(t)$ for some family of rate matrices $\hat{L}(\cdot)$, with initial condition $\hat{\mu}(0) = \xi$, terminal condition $\lim_{t \ra +\infty} \hat{\mu}(t) = \xi_0$, and $\hat{\mu}(t) \in \MA_1(\ZA)$ for all $t \geq 0$.
\end{lemma}

\begin{proof}
The path given by Lemma \ref{lem:one-trajectory} converges as shown in Lemma \ref{lem:converge-w-limit} to $\Omega$, which is contained in an $\omega$-limit set of the ODE (\ref{eqn:backward-mckvla-dynamics-actual}). So $\Omega$ is connected and compact. Since the path $\hat{\mu}(\cdot)$ also stays completely within $\MA_1(\ZA)$ (viewed as a subset of $\mathbb{R}^r$), $\Omega$ is a subset of $\MA_1(\ZA)$. Furthermore, $\Omega$ is invariant (both positively and negatively) to the dynamics defined by the ODE (\ref{eqn:backward-mckvla-dynamics-actual}). But then, by our assumption that $\xi_0$ is the unique globally asymptotically stable equilibrium for the McKean-Vlasov dynamics
\begin{equation}
  \label{eqn:forward-dynamics}
  \dot{\mu}(t) = A_{\mu(t)}^* \mu(t),
\end{equation}
we must have $\Omega = \{ \xi_0 \}$ because this is the only subset of $\MA_1(\ZA)$ that is invariant to both of the dynamics in (\ref{eqn:backward-mckvla-dynamics-actual}) and (\ref{eqn:forward-dynamics}). Letting $m \ra +\infty$ in (\ref{eqn:equality-in-parts}), it follows by lower semicontinuity of $s$ that
\begin{equation}
    \label{eqn:s-geq}
    s(\xi) \geq s(\xi_0) + \int_{[0, +\infty)} \Big[ \sum_{(i,j) \in \EA} (\hat{\mu}(t)(i)) \hat{\lambda}_{i,j}(\hat{\mu}(t))
             \tau^* \left( \frac{\hat{l}_{i,j}(t)}{\hat{\lambda}_{i,j}(\hat{\mu}(t))} - 1 \right) \Big]~dt.
\end{equation}
By Lemma \ref{lem:ugase-basevalue}, $s(\xi_0) = 0$, and thus
\[
    s(\xi) \geq \int_{[0, +\infty)} \Big[ \sum_{(i,j) \in \EA} (\hat{\mu}(t)(i)) \hat{\lambda}_{i,j}(\hat{\mu}(t))
             \tau^* \left( \frac{\hat{l}_{i,j}(t)}{\hat{\lambda}_{i,j}(\hat{\mu}(t))} - 1 \right) \Big] ~dt
\]
for the special path $\hat{\mu}$. This establishes that $s(\xi)$ is at least the right-hand side of (\ref{eqn:s}). By an application of (\ref{eqn:HJB}) with $\nu = \xi_0$, it is obvious that $s(\xi)$ is upper bounded by the right-hand side of (\ref{eqn:s}), whence equality holds in (\ref{eqn:s}) and the uniqueness of $s(\xi)$ follows.
\end{proof}

We are now ready to prove Theorem \ref{thm:invariant}.

\begin{proof}[Proof of Theorem \ref{thm:invariant}]
Take any arbitrary sequence of natural numbers growing to $+\infty$. By Lemma \ref{lem:HJB}, there is a subsequence that satisfies the large deviation principle with rate function $s$ such that (\ref{eqn:HJB}) holds. Given our assumption that $\xi_0$ is the unique globally asymptotically stable equilibrium for the McKean-Vlasov equation, we have $s \geq 0$, and $s(\xi_0) = 0$ by Lemma \ref{lem:ugase-basevalue}. By Lemma \ref{lem:uniqueness}, $s$ is uniquely specified by (\ref{eqn:s}), which is the same as (\ref{eqn:s-2}). Thus every sequence contains a further subsequence $(N_k, k \geq 1)$ such that $(\wp^{(N_k)}, k \geq 1)$ satisfies the large deviation principle with speed $N_k$ and the same rate function $s$ specified by (\ref{eqn:s-2}). By \cite[Ex. 4.4.15(a)-(b), pp. 147-148]{DZ}, it follows that $(\wp^{(N)}, N \geq 1)$ satisfies the large deviation principle with speed $N$ and rate function given by (\ref{eqn:s-2}).
\end{proof}

\subsection{Invariant measure: The general case}
\label{subsec:Inv-measure-multiple-eq}

We now treat the general case under assumption ({\bf B}). As noted earlier, under the hypothesis of Theorem \ref{thm:invariant}, assumption ({\bf B}) holds with $l=1$, $K_1 = \{ \xi_0 \}$, and moreover $s(\xi_0) = 0$ by Lemma \ref{lem:ugase-basevalue}. We generalize Lemma \ref{lem:ugase-basevalue} in Lemma \ref{lem:s-multiple-equilibria} by first proving the following lemma. The generalization of Lemma \ref{lem:uniqueness} is Lemma \ref{lem:uniqueness-multiple-equilibria}.

\begin{lemma}
\label{lem:multiple-equilibria-basevalue}
Under assumption {\em ({\bf B})}, the rate function $s$ satisfies the following:
\begin{itemize}
  \item There exists $\xi_0 \in K_{i_0}$ for some $i_0 = 1, 2, \ldots, l$ that satisfies $s(\xi_0) = 0$.
  \item There exist nonnegative real numbers $s_1, s_2, \ldots, s_l$ such that $\xi \in K_i$ implies $s(\xi) = s_i$.
\end{itemize}
\end{lemma}

\begin{proof}
The first statement immediately follows from the steps in the proof of Lemma \ref{lem:ugase-basevalue}: $\xi_0$ is now some element in the $\omega$-limit set for the McKean-Vlasov dynamics with initial condition $\nu^*$ satisfying $s(\nu^*) = 0$.

For the second statement, let $\nu, \xi \in K_i$. Fix $\varepsilon > 0$. Since $\nu \sim \xi$, there exists $T>0$ such that $S_T(\xi | \nu) \leq \varepsilon$. Using (\ref{eqn:HJB}), we get
\[
  s(\xi) \leq s(\nu) + S_T(\xi | \nu) \leq s(\nu) + \varepsilon.
\]
Reversing the role of $\xi$ and $\nu$, we get $s(\nu) \leq s(\xi) + \varepsilon$, whence $|s(\xi) - s(\nu)| \leq \varepsilon$. Since $\varepsilon$ was arbitrary, we must have $s(\nu) = s(\xi)$. So all points in the compact set  $K_i$ take the same value. The second statement follows.
\end{proof}

We now argue that the function $s$ that satisfies (\ref{eqn:HJB}), (\ref{eqn:equality-in-parts}), the condition $s \geq 0$, and the condition $\min_{\nu} s(\nu) = 0$ is indeed unique. We do this in two steps via the following lemmas.

\begin{lemma}
\label{lem:uniqueness-multiple-equilibria}
Assume that {\em ({\bf B})} holds. Let $s_1, s_2, \ldots, s_l$ be specified as the values on the compact sets $K_1, K_2, \ldots, K_l$. Let $s_{l'} \geq 0$ for $1 \leq l' \leq l$ and let $\min \{ s_1, s_2, \ldots, s_l \} = 0$. Then the solution to (\ref{eqn:HJB}) and (\ref{eqn:equality-in-parts}) is unique and is given by
  \begin{equation}
  \label{eqn:s-multiple-equilibria}
    s(\xi) = \inf_{l'} \inf_{\hat{\mu}} \left[ s_{l'}
             + \int_0^{+\infty} \Big[ \sum_{(i,j) \in \EA} (\hat{\mu}(t)(i)) \hat{\lambda}_{i,j}(\hat{\mu}(t))
               \tau^* \left( \frac{\hat{l}_{i,j}(t)}{\hat{\lambda}_{i,j}(\hat{\mu}(t))} - 1 \right) \Big] ~dt \right]
  \end{equation}
where the second infimum is over all $\hat{\mu}$ that are solutions to the dynamical system $\dot{\hat{\mu}}(t) = -\hat{L}(t)^* \hat{\mu}(t)$ for some family of rate matrices $\hat{L}(\cdot)$, with initial condition $\hat{\mu}(0) = \xi$, terminal condition $\hat{\mu}(t) \ra K_{l'}$ as $t \ra +\infty$, and $\hat{\mu}(t) \in \MA_1(\ZA)$ for all $t \geq 0$.
\end{lemma}

\begin{proof}
The same steps of the proof of Lemma \ref{lem:uniqueness} apply with the following modifications. $\Omega = \{ \xi_0 \}$ gets replaced by $\Omega \subset K_{l'}$ for some $l'$. Consequently, in the lower bound in (\ref{eqn:s-geq}), $s(\xi_0)$ gets replaced by $s_{l'}$.
\end{proof}

Recall that we defined $s_1, s_2, \ldots, s_l$ in (\ref{eqn:s-values}). We now assert the following.

\begin{lemma}
\label{lem:s-multiple-equilibria}
The rate function $s$ has values $s_1, s_2, \ldots, s_l$ given by (\ref{eqn:s-values}).
\end{lemma}

\begin{proof}
Immediate from Theorem \ref{thm:FW-6.4.1} in Appendix.
\end{proof}

\begin{Remark}
 In order to obtain these values, one has to go beyond the ODE method. One has to consider the empirical measure Markov process sampled at hitting times of neighborhoods of the compact sets. This is done in Freidlin and Wentzell \cite[Ch. 6]{Freidlin} for diffusions on a compact manifold, with $V$ satisfying a Lipschitz property. Thanks to Corollary \ref{cor:uniform-ldp-flows}, Lemma \ref{lem:S_T-bounding}, Lemma \ref{lem:uniform-continuity}, the same program can be carried out with straightforward modifications to account for the fact that we have to handle jumps and the fact that the minimum cost function $V(\cdot | \cdot)$ satisfies only the uniform continuity property. The appendix provides the necessary verification.
\end{Remark}

With this disambiguation for the values of $s$ at the compact sets $K_i$, we now ready to finish the proof of Theorem \ref{thm:invariant-multiple-equilibria}.

\begin{proof}[Proof of Theorem \ref{thm:invariant-multiple-equilibria}]
Same as that of Theorem \ref{thm:invariant}, but with the use of Lemmas \ref{lem:uniqueness-multiple-equilibria} and \ref{lem:s-multiple-equilibria}.
\end{proof}

The following sections complete the proofs of some assertions of Section~{\ref{sec:ldp-finite-time}}.

\section{Proof of Theorem \ref{thm:empirical-measure}}
\label{sec:Proof-empirical-measure}

The proof is based on a generalization of Sanov's theorem due to Dawson and G\"artner \cite{Dawson-Gartner}, the Girsanov transformation, and the Laplace-Varadhan principle. We proceed through a sequence of lemmas.

\subsection{The noninteracting case}
Consider first the noninteracting case.

\begin{lemma}
  \label{lem:non-interacting-empirical}
  Suppose that the initial conditions $\nu_N \ra \nu$ weakly. Then the sequence
  $ (P^{o, (N)}_{\nu_N} , N \geq 1 )$ satisfies the large deviation principle in $\MA_{1, \varphi}(\XA)$, endowed with the weak* topology $\sigma(\MA_{1, \varphi}(\XA), C_{\varphi}(\XA))$, with speed $N$ and good rate function $J(Q)$ given by (\ref{eqn:J-hat(Q)}).
\end{lemma}

\begin{proof}
The family $\{ P_z, z \in \ZA \}$ is clearly a subset of $\MA_{1, \varphi}(\XA)$ since the transition rate from any state to any of the other (at most $r$) states is upper bounded by 1. The family $\{ P_z, z \in \ZA \}$ is also Feller continuous in the discrete topology on $\ZA$. Since $\nu_N \rightarrow \nu$, by Dawson and G\"artner's \cite[Th. 3.5]{Dawson-Gartner} which is a generalization of Sanov's theorem, we have that $(P^{o, (N)}_{\nu_N}, N \geq 1)$ satisfies the large deviation principle in the weak* topology $\sigma(\MA_{1,\varphi}(\XA), C_{\varphi}(\XA))$ with speed $N$ and good rate function $J(Q)$ given by (\ref{eqn:J-hat(Q)}).
\end{proof}

Our next lemma states that $\MA_{1,\varphi}(\XA)$ contains all the probability measures of interest to us.

\begin{lemma}
  \label{lem:J(Q)-bound}
  If $J(Q) < +\infty$ then (1) $Q \in \MA_{1,\varphi}(\XA)$ and (2) $Q \circ \pi_0^{-1} = \nu$.
\end{lemma}

\begin{proof}
Observe that $||\varphi||_{\varphi} \leq 1$, and the topology on $\XA$ is chosen so that $\varphi$ is continuous. Hence $\varphi \in C_{\varphi}(\XA)$. Recalling the expression for $J(Q)$ in (\ref{eqn:J-hat(Q)}), we get that $J(Q) < +\infty$ implies
\begin{equation}
  \label{eqn:finite-J(Q)}
  \int_{\XA} \varphi ~dQ - \sum_{z \in \ZA} \nu(z) \log \int_{\XA} e^{\varphi} ~dP_z < +\infty.
\end{equation}
Since the transition rates for $P_z$ are upper bounded by 1, and there are at most $r$ possibilities for jumps from any state, $\varphi$ is stochastically dominated by a Poisson random variable with parameter $rT$. Consequently $1 \leq \int_{\XA} e^{\varphi} ~dP_z < +\infty$ for each $z \in \ZA$, and it follows from (\ref{eqn:finite-J(Q)}) that $\int_{\XA} \varphi ~dQ < +\infty$, and so $Q \in \MA_{1, \varphi}(\XA)$.

To prove the second conclusion, suppose $\nu_Q = Q \circ \pi_0^{-1} \neq \nu$. Consider bounded functions $f(x) = f_0(\pi_0(x))$ that depend on $x$ only through the initial condition. Since $\nu_Q \neq \nu$, there exists a function $f$ of the above form that also satisfies $\sum_z f_0(z) \nu_Q(z) - \sum_z f_0(z) \nu(z) \neq 0$. By flipping the sign of $f$ if necessary and by scaling, we may assume that $\sum_z f_0(z) \nu_Q(z) - \sum_z f_0(z) \nu(z) = a$ for an arbitrary $a \in (0, +\infty)$. This $f$ is bounded continuous and hence is in $C_{\varphi}(\XA)$. A simple calculation yields
\[
  \int_{\XA} f ~dQ - \sum_{z \in \ZA} \nu(z) \log \int_{\XA} e^f ~ dP_z = \sum_z f_0(z) \nu_Q(z) - \sum_z f_0(z) \nu(z) = a.
\]
Since $a > 0$ was arbitrary, $J(Q) = +\infty$, and the second part is proved by contraposition.
\end{proof}

We next get an alternative expression for $J(Q)$ as a supremum over the more familiar space $C_b(\XA)$ of bounded continuous functions on $\XA$.

\begin{lemma}
  \label{lem:J(Q)-variational}
  For each $Q \in \MA_{1, \varphi}(\XA)$, $J(Q)$ defined in (\ref{eqn:J-hat(Q)}) admits the alternative characterization
  \begin{equation}
    \label{eqn:J(Q)-2}
    J(Q) = \sup_{f \in C_b(\XA)} \left[ \int_{\XA} f ~dQ - \sum_{z \in \ZA} \nu(z) \log \int_{\XA} e^f ~ dP_z \right].
  \end{equation}
\end{lemma}

\begin{proof}
In (\ref{eqn:J(Q)-2}), the supremum is taken over $C_b(\XA)$, the set of bounded continuous functions on $\XA$, while in (\ref{eqn:J-hat(Q)}) the supremum is over $C_{\varphi}(\XA)$.

Fix $Q \in \MA_{1,\varphi}(\XA)$. Since $C_b(\XA) \subset C_{\varphi}(\XA)$, $J(Q)$ defined by (\ref{eqn:J-hat(Q)}) is at least the right-hand side of (\ref{eqn:J(Q)-2}). For the other direction, let $f \in C_{\varphi}(\XA)$, and consider the truncations $f_n$ of the function $f$ to $[-n, n]$. Clearly $\{ f_n \} \subset C_b(\XA)$, and $f_n \ra f$ a.e.-$[Q]$ and a.e.-$[P_z]$. Since we also have $\int_{\XA} f ~dQ < +\infty$ and $\int_{\XA} e^f ~dP_z < +\infty$, both of which can be easily checked using $||f||_{\varphi} < +\infty$, an application of the Lebesgue dominated convergence theorem yields
\begin{eqnarray*}
  \lefteqn{ \lim_{n \ra +\infty} \left[ \int_{\XA} f_n ~dQ - \sum_{z \in \ZA} \nu(z) \log \int_{\XA} e^{f_n} ~ dP_z \right] } \\
  & = & \int_{\XA} f ~dQ - \sum_{z \in \ZA} \nu(z) \log \int_{\XA} e^f ~ dP_z.
\end{eqnarray*}
Since $f \in C_{\varphi}(\XA)$ was arbitrary, $J(Q)$ in (\ref{eqn:J-hat(Q)}) is at most the right-hand side of (\ref{eqn:J(Q)-2}).
\end{proof}

We now characterize $J(Q)$ further. We begin by getting a lower bound.

\begin{lemma}
\label{lem:J(Q)-lower}
Let $Q$ be such that $Q \circ \pi_0^{-1} = \nu$. Define $P$ as in (\ref{eqn:P}). We then have $H(Q|P) \leq J(Q)$.
\end{lemma}

\begin{proof}
For any $f \in C_b(\XA)$ and with $P$ as defined in (\ref{eqn:P}), Jensen's inequality yields
\[
  \int_{\XA} f ~dQ - \sum_z \nu(z) \log \int_{\XA} e^f ~dP_z \geq \int_{\XA} f ~dQ - \log \int_{\XA} e^f ~dP.
\]
By Lemma \ref{lem:J(Q)-variational}, $J(Q)$ is the supremum of the left-hand side over all $f \in C_b(\XA)$. By the variational formula for relative entropy (\cite[Lem. 6.2.13]{DZ}), $H(Q|P)$ is the supremum of the right-hand side over all $f \in C_b(\XA)$. This establishes the inequality.
\end{proof}

We next show that $J(Q)$ is upper bounded by another relative entropy. To do this, let us introduce the Polish space $(\hat{\XA}, \hat{d})$ where $\hat{\XA} = \ZA \times \XA$ and the metric $\hat{d}$ is
\[
  \hat{d}((i,x), (j,y)) = {\bf 1}_{ \{i \neq j\} } + d(x,y),
\]
where $d$ is the metric on $\XA$. The first component of $\hat{\XA}$ shall denote the initial condition. For two measures $\hat{R}_1$ and $\hat{R}_2$ on $\hat{\XA}$, let $H(\hat{R}_1 | \hat{R}_2)$ be the relative entropy of $\hat{R}_1$ with respect to $\hat{R}_2$.

For a fixed $\nu \in \MA_1(\ZA)$, let us now define $\hat{P}$ as
\begin{equation}
  \label{eqn:P-hat-defn}
  d\hat{P}(z,x) = \nu(z) dP_z(x).
\end{equation}
The push forward of $\hat{P}$ under the projection mapping $(z,x) \in \hat{\XA} \mapsto x \in \XA$ is clearly the $P$ defined in (\ref{eqn:P}).

Let $Q$ be such that $Q \circ \pi_0^{-1} = \nu$. We then define $\hat{Q}$ as
\begin{equation}
  \label{eqn:Q-hat-defn}
  d\hat{Q}(z,x) = dQ(x) {\bf 1}_{\{\pi_0(x)\}}(z).
\end{equation}
Observing that $\hat{\XA}$ is Polish and that the push forward of $\hat{Q}$ under the projection $(z,x) \mapsto z$ is $\nu$, it follows that there is a regular conditional probability measure $Q_z$ satisfying
\begin{equation}
  \label{eqn:Q-hat-alt}
  d\hat{Q}(z,x) = \nu(z) dQ_z(x).
\end{equation}
Putting (\ref{eqn:Q-hat-defn}) and (\ref{eqn:Q-hat-alt}) together and summing over $z$, we obtain that the second marginal of $\hat{Q}$ is
\begin{equation}
    \label{eqn:Q-marginal}
    dQ(x) = \sum_{z} dQ(x) {\bf 1}_{\{\pi_0(x)\}}(z) = \sum_z \nu(z) dQ_z(x).
\end{equation}

Since both $\hat{Q}$ and $\hat{P}$ have the same first marginal $\nu$, the decomposition result for relative entropy \cite[Th. D.8]{DZ} gives
\begin{equation}
  \label{eqn:relative-entropy-hat}
  H(\hat{Q} | \hat{P}) = \left\{
             \begin{array}{cl}
               \displaystyle \sum_{z \in \ZA} \nu(z) H(Q_z | P_z) & \mbox{if } Q_z \ll P_z \mbox{ a.e.-}[\nu] \\
               +\infty & \mbox{otherwise}.
             \end{array}
           \right.
\end{equation}
With these preliminaries, we are now ready to upper bound $J(Q)$.

\begin{lemma}
\label{lem:J(Q)-u-l}
Let $Q$ be such that $Q \circ \pi_0^{-1} = \nu$. With $\hat{P}$ and $\hat{Q}$ as above, we have $J(Q) \leq H(\hat{Q} | \hat{P})$.
\end{lemma}

\begin{proof}
For $f \in C_b(\XA)$, by (\ref{eqn:Q-marginal}), we have $\int_{\XA} f ~dQ = \sum_{z} \nu(z) \int_{\XA} f ~dQ_z$, and so
\begin{eqnarray*}
  \lefteqn{ \int_{\XA} f ~dQ - \sum_z \nu(z) \log \int_{\XA} e^f ~dP_z } \\
  & = & \sum_z \nu(z) \left[ \int_{\XA} f ~dQ_z - \log \int_{\XA} e^f ~dP_z \right] \\
  & \leq & \sum_z \nu(z) \left( \sup_{f \in C_b(\XA)} \left[ \int_{\XA} f ~dQ_z - \log \int_{\XA} e^f ~dP_z \right] \right) \\
  & = & H(\hat{Q} | \hat{P}),
\end{eqnarray*}
where the last equality follows from the observation that the term within parenthesis in the immediately preceding inequality is the variational representation for the relative entropy $H(Q_z | P_z)$, and from (\ref{eqn:relative-entropy-hat}). Take supremum over all $f \in C_b(\XA)$ and use Lemma \ref{lem:J(Q)-variational} to deduce that $J(Q)$ is upper bounded by $H(\hat{Q} | \hat{P})$.
\end{proof}

\begin{lemma}
  \label{lem:J(Q)is-entropy}
  Let $Q \in \MA_{1,\varphi}(\XA)$. We then have
  \[
    J(Q) = \left\{
             \begin{array}{cl}
               H(Q | P), & \mbox{ if } Q \circ \pi_0^{-1} = \nu \\
               +\infty & \mbox{otherwise.}
             \end{array}
           \right.
  \]
\end{lemma}

\begin{proof}
If $Q \circ \pi_0^{-1} \neq \nu$, by Lemma \ref{lem:J(Q)-bound}, we have $J(Q) = +\infty$. So assume $Q \circ \pi_0^{-1} = \nu$. By Lemmas \ref{lem:J(Q)-lower} and \ref{lem:J(Q)-u-l}, we get
\[
  H(Q | P) \leq J(Q) \leq H(\hat{Q} | \hat{P})
\]
where $\hat{P}$ and $\hat{Q}$ are defined in (\ref{eqn:P-hat-defn}) and (\ref{eqn:Q-hat-defn}), respectively. The second marginal of $\hat{P}$ is $P = \sum_z \nu(z) P_z$. From the definition of $\hat{P}$ it is clear that $d\hat{P}(z,x) = dP(x) {\bf 1}_{\{\pi_0(x)\}}(z)$, so that the regular conditional probability measures of both $\hat{P}$ and $\hat{Q}$, given the second component $x$, are the same, that is, $d\hat{Q}(z|x) = d\hat{P}(z|x) = {\bf 1}_{\{\pi_0(x)\}}(z)$ a.e.-$[Q]$. In particular, $H(\hat{Q}(\cdot|x) | \hat{P}(\cdot|x)) = 0$ a.e.-$[Q]$. By the decomposition result for relative entropy \cite[Th. D.8]{DZ}, we get
\begin{eqnarray*}
  H(\hat{Q} | \hat{P}) & = & H(Q | P) + \int_{\XA} H(\hat{Q}(\cdot|x) | \hat{P}(\cdot|x)) ~dQ(x) \\
  & = & H(Q | P).
\end{eqnarray*}
The lemma follows.
\end{proof}

\subsection{Continuity of the function $h(Q)$}
\label{subsec:hQ-continuity}

We now proceed to address some preliminaries required for the interacting case. The Radon-Nikodym derivative (\ref{eqn:RND}) of $P^{(N)}_{\nu_N}$ with respect to $P^{o,(N)}_{\nu_N}$ is
\[
  \frac{dP^{(N)}_{\nu_N}}{dP^{o,(N)}_{\nu_N}}(Q) = \exp \{ N h(Q) \}.
\]
We now study the continuity property of $h(Q)$. Towards this, we first establish a regularity property for all $Q$ with $J(Q) < +\infty$. We then appeal to results of \cite{Leonard} to establish the continuity of $h(Q)$ when $J(Q) < +\infty$.

\begin{lemma}
\label{lem:sup-jump-condition}
Let $J(Q) < +\infty$ and suppose that the random variable $X$ is distributed according to $Q$. Then
\begin{equation}
  \label{eqn:sup-jump-condition}
  \sup_{t \in [0,T]} \mathbb{E} \left[ \sup_{u \in [t-\alpha, t+\alpha] \cap [0,T]} \{ {\bf 1}_{\{X_u \neq X_{u-}\}} \} \right] \ra 0 \mbox{ as } \alpha \downarrow 0.
\end{equation}
\end{lemma}

\begin{proof}
A proof for the case when $Q \circ \pi_0^{-1} = \delta_{z_0}$ for some fixed $z_0$ can be found in \cite[eqn. (2.14), p. 309-310]{Leonard}. Our argument below is a simple modification and relies only on what we have thus far established for $J(Q)$.

Let $K = \{ x \in \XA : \exists u \in [t-\alpha, t+\alpha] \cap [0,T] \mbox{ satisfying } x_u \neq x_{u-} \}$. Since $J(Q) < +\infty$, it follows that $Q \ll P$. We may therefore write
\begin{eqnarray}
  \mathbb{E} \left[ \sup_{u \in [t-\alpha, t+\alpha] \cap [0,T]} \{ {\bf 1}_{ \{ X_u \neq X_{u-} \}} \} \right] & = & Q (K) \nonumber \\
    & = & \int_{\XA} \left( \frac{dQ}{dP} \right)~{\bf 1}_K ~dP \nonumber \\
    \label{eqn:Orlicz-Holder}
    & \leq & \left\| \frac{dQ}{dP} \right\|_{\tau^*, P} ~ \left\| {\bf 1}_K \right\|_{\tau, P}
\end{eqnarray}
where $|| f ||_{\tau^*, P}$ is the Orlicz norm
\[
  || f ||_{\tau^*, P} = \inf \left\{ a>0 : \int_{\XA} \tau^* \left( \frac{|f(x)|}{a} \right) ~dP(x) \leq 1 \right\}
\]
(with respect to the function $\tau^*$ and measure $P$), $|| g ||_{\tau, P}$ is a similarly defined Orlicz norm with respect to the function $\tau$ and measure $P$, and the inequality in (\ref{eqn:Orlicz-Holder}) is the H\"older inequality in Orlicz spaces. See the Appendix in \cite{Leonard} for a summary of key results on Orlicz spaces.

The lemma's proof will be complete if we can show that $\big\| \frac{dQ}{dP} \big\|_{\tau^*, P}$ is bounded, and $\left\| {\bf 1}_K \right\|_{\tau, P}$ vanishes as $\alpha \downarrow 0$. We proceed to justify these claims.

$J(Q) < +\infty$ implies $\big\| \frac{dQ}{dP} \big\|_{\tau^*, P} < +\infty$. Indeed, since
\[
  \lim_{u \ra +\infty} \frac{\tau^*(u)}{u \log u} = 1,
\]
choose a large enough $u_0 > 0$ such that $\tau^*(u) \leq 2 u \log u$ for $u \geq u_0$. Also observe that $\tau^*(u)$ is increasing in $u$ and $u
\log u \geq - e^{-1}$ for $u \geq 0$. Thus with $f = dQ/dP$,
\begin{eqnarray}
  \int_{\XA} \tau^*(f) ~dP & \leq & \tau^*(u_0) + \int_{\{x \in \XA: f(x) \geq u_0 \}} 2 f \log f ~dP \nonumber \\
  & \leq & \tau^*(u_0) + 2 J(Q) + 2e^{-1} \nonumber \\
  \label{eqn:Orlicz-bound}
  & < & +\infty.
\end{eqnarray}
Since $\tau^*$ is convex and $\tau^*(0) = 0$, Jensen's inequality yields
\[
  \tau^*(f/a) \leq \tau^*(f)/a \mbox{ for } a \geq 1.
\]
This fact in conjunction with (\ref{eqn:Orlicz-bound}) implies that $||f||_{\tau^*,P} < +\infty$.

Now consider $\left\| {\bf 1}_K \right\|_{\tau, P}$. Since $\tau(0) = 0$, we get $\tau(({\bf 1}_K)/a) = \tau(1/a) {\bf 1}_K$, and so
\[
  \int_{\XA} \tau \left( \frac{{\bf 1}_K}{a} \right) ~dP = \tau(1/a) \int_{\XA} {\bf 1}_K ~dP = \tau(1/a) P(K).
\]
Under $P$, the transition rates are upper bounded by 1. Moreover, there are at most $r$ possible next states. Since $K$ is the event that there is a transition in $[t-\alpha, t+\alpha] \cap [0,T]$, it follows that $P(K) \leq 2\alpha r$. From its definition, the Orlicz norm is the smallest positive $a$ such that $\tau(1/a) P(K) \leq 1$, and so
\[
  \left\| {\bf 1}_K \right\|_{\tau, P} = \frac{1}{\tau^{-1}(1/P(K))} \leq \frac{1}{\tau^{-1}(1/(2\alpha r))}.
\]
where the equality in the above chain holds because $\tau(u)$ is increasing in $u$ for $u \geq 0$, and the inequality holds because of the same property for $\tau^{-1}(u)$. This last upper bound vanishes as $\alpha \downarrow 0$.
\end{proof}

The following lemma implies Lemma \ref{lem:continuity-pi} as a corollary.

\begin{lemma}
\label{lem:continuity-pi-appendix}
Consider $\MA_1(D([0,T], \ZA))$ endowed with the topology of weak convergence and $D([0,T], \MA_1(\ZA))$ endowed with the topology induced by the metric $\rho_T$ defined in (\ref{eqn:metric-flow-space}). Then the mapping
\[
  \pi: \MA_1(D([0,T], \ZA)) \ra D([0,T], \MA_1(\ZA))
\]
is continuous at each $Q \in \MA_1(D([0,T], \ZA))$ where $J(Q) < +\infty$.
\end{lemma}

\begin{proof}
Note that, by Lemma \ref{lem:J(Q)-bound}, $J(Q) < +\infty$ implies $Q \in \MA_{1,\varphi}(\XA)$. The statement of the above Lemma is the same as \cite[Lem. 2.8]{Leonard}. The only difference is that our representation for $J(Q)$ differs in order to handle nonchaotic initial conditions and allows $Q \circ \pi_0^{-1}$ to be any measure in $\MA_1(\ZA)$.

The proof of \cite[Lem. 2.8]{Leonard} holds verbatim if we can establish (\ref{eqn:sup-jump-condition}) (which is the same as \cite[eqn. (2.14)]{Leonard}) for the more general case under consideration. This is done in Lemma \ref{lem:sup-jump-condition}.
\end{proof}

This is an appropriate location to include the proof of Lemma \ref{lem:continuity-pi}, which is a corollary to the above Lemma.

\begin{proof}[Proof of Lemma \ref{lem:continuity-pi}]
The first part is a corollary to Lemma \ref{lem:continuity-pi-appendix}. Indeed, Lemma \ref{lem:continuity-pi-appendix} shows that $\pi$ is continuous under the coarser topology of weak convergence of probability measures metrized by $d_{\textsf{Sko}}$. Since the natural embedding of $\XA$ into $D([0,T], \ZA)$ (with topology induced by $d_{\textsf{Sko}}$) is continuous, it immediately follows that $\pi$ is a continuous mapping under the finer topology $\sigma(\MA_{1,\varphi}(\XA), C_{\varphi}(\XA))$.

To see the second part, fix a $Q$ such that $J(Q) < +\infty$, a $t \in [0,T]$, and consider a sequence $Q_N \ra Q$. By the first part, we have $\pi(Q_N) \rightarrow \pi(Q)$, which is the same as saying $\rho_T(\pi(Q_N), \pi(Q)) \ra 0$. But then
\[
  \rho_0(\pi_t(Q_N), \pi_t(Q)) \leq \rho_T(\pi(Q_N), \pi(Q)) \ra 0
\]
establishes the continuity of $\pi_t$.
\end{proof}

We now come to the continuity of the function $h(Q)$.

\begin{lemma}
\label{lem:h-continuous}
Consider the space $\MA_{1,\varphi}(\XA)$ endowed with the weak* topology $\sigma(\MA_{1,\varphi}(\XA), C_{\varphi}(\XA))$. The function $h: \MA_{1,\varphi}(\XA) \ra \mathbb{R}$ defined in (\ref{eqn:h(Q)}) is continuous at every $Q$ where $J(Q) < +\infty$.
\end{lemma}

\begin{proof}
The statement is the same as \cite[Lem. 2.9]{Leonard}. The same proof applies. That proof requires continuity of $\pi$, which is now established in Lemma \ref{lem:continuity-pi-appendix} under assumptions ({\bf A1})-({\bf A3}).
\end{proof}

\subsection{The interacting case}

We now address the interacting case.

\begin{proof}[Proof of Theorem \ref{thm:empirical-measure}]

Recall the statement of Theorem \ref{thm:empirical-measure}. We are now
given that the sequence of initial empirical measures $\nu_N \ra \nu$
weakly. By Lemma \ref{lem:non-interacting-empirical}, $(P^{o,(N)}_{\nu_N}, N \geq 1)$
satisfies the large deviation
principle in the topological space $\MA_{1,\varphi}(\XA)$ with rate
function $J(Q)$. By Lemma \ref{lem:J(Q)-bound} and Lemma
\ref{lem:h-continuous}, $h$ is continuous on the set $\{ Q \in
\MA_{1,\varphi}(\XA) \mid J(Q) < +\infty \}$. Furthermore, by
\cite[Lem. 2.10]{Leonard}, for every $\alpha > 0$, we have
\[
  \limsup_{N \ra +\infty} \frac{1}{N} \log \int_{\MA_{1,\varphi}(\XA)} e^{N\alpha |h|} ~dP^{o,(N)}_{\nu_N} < +\infty.
\]
Using the Laplace-Varadhan principle, see \cite[Prop. 2.5]{Leonard}, we can draw two conclusions. The first conclusion is that
\begin{equation}
  \label{eqn:varadhan-lemma}
  \frac{1}{N} \log \int_{\MA_{1,\varphi}(\XA)} e^{Nh} ~dP^{o,(N)}_{\nu_N} \ra \sup_{Q' \in \MA_{1,\varphi}(\XA)} [h(Q') - J(Q')]
\end{equation}
as $N \ra +\infty$. From (\ref{eqn:RND}), we have
\[
  e^{Nh} dP^{o,(N)}_{\nu_N} = dP^{(N)}_{\nu_N},
\]
a probability measure. The left-hand side in (\ref{eqn:varadhan-lemma}) is therefore always 0, and so $\sup_{Q' \in \MA_{1,\varphi}(\XA)} [h(Q') - J(Q')] = 0$.
The second conclusion is that $( P^{(N)}_{\nu_N}, N \geq 1)$ satisfies the large deviation principle in the topological space $\MA_{1,\varphi}(\XA)$ with good rate function
\[
  I(Q) = J(Q) - h(Q) - \inf_{Q'} [J(Q') - h(Q')] = J(Q) - h(Q)
\]
where the last equality holds because the infimum above is 0 by the first conclusion. This concludes the proof of the first part of Theorem \ref{thm:empirical-measure}.

\enlargethispage{3pt}
We now show (\ref{eqn:I(Q)}). By assumption ({\bf A1})-({\bf A3}), it is easy to see that there exists a constant $K$ such that $|h(Q)| \leq K (1 + \int_{\XA} \varphi ~dQ)$ so that if $Q \in \MA_{1,\varphi}(\XA)$ then $|h(Q)| < +\infty$. By Lemma \ref{lem:J(Q)is-entropy}, if either $Q \circ \pi_0^{-1} \neq \nu$ or $Q$ is not absolutely continuous with respect to $P$, then $J(Q) = +\infty$, and by the finiteness of $h(Q)$, we have $I(Q) = J(Q) - h(Q) = +\infty$. We may therefore assume $Q \circ \pi_0^{-1} = \nu$ and $Q \ll P$, whence, by Lemma \ref{lem:J(Q)is-entropy} once again, $J(Q) = H(Q|P) = \int dQ ~\log (dQ/dP)$. It therefore suffices to argue that
\[
  Q \circ \pi_0^{-1} = \nu \mbox{ and }  Q \ll P \quad \Rightarrow \quad H(Q | P) - h(Q) = H(Q | P(\pi(Q))).
\]
Let $\mu = \pi(Q)$ for convenience. Observe that the density
\[
  \frac{dP_z(\mu)}{dP_z}(\cdot) = \exp\{ h_1(\cdot, \mu) \}
\]
in (\ref{eqn:RND-pathlevel}) does not depend on $z$. It follows that the density of the mixture distribution $P(\mu)$ in (\ref{eqn:P(mu)}) with respect to the mixture $P$ in (\ref{eqn:P}) is
\[
  \frac{dP(\mu)}{dP}(x) = \exp \{ h_1(x,\mu) \}.
\]
Using this in (\ref{eqn:h(Q)}), we get
\[
  h(Q) = \int_{D([0,T], \ZA)} dQ ~\log \frac{dP(\mu)}{dP},
\]
from which
\begin{eqnarray*}
  H(Q|P) - h(Q) & = & \int_{D([0,T], \ZA)} dQ ~\log \frac{dQ}{dP} - \int_{D([0,T], \ZA)} dQ ~\log \frac{dP(\mu)}{dP} \\
  & = & \int_{D([0,T], \ZA)} dQ ~\log \frac{dQ}{dP(\mu)} \\
  & = & H(Q | P(\mu))
\end{eqnarray*}
follows. This concludes the proof.
\end{proof}
\eject

\section{Proof of Lemma \ref{lem:S_T-bounding}}
\label{sec:ProofOfS_T-bounding}

\allowdisplaybreaks
We address the first bullet. For ease of exposition, let us for now allow all possible transitions and ignore the constraint that only $\EA$ transitions are allowed. Consider the constant velocity path
\begin{equation}
  \label{eqn:const-velocity-path}
  \mu(t) = \left( 1 - \frac{t}{T} \right) \nu + \frac{t}{T} \xi, ~ t \in [0,T]
\end{equation}
for which
\[
  \dot{\mu}(t) = \frac{\xi - \nu}{T}, ~t \in [0,T].
\]
There is flow out of $i$ if $\xi(i) < \nu(i)$, and flow into $i$ otherwise. We now construct a rate matrix $L(t)$ with entries $l_{i,j}(t)$ that ensure the traversal of this constant velocity path.

Since there is conservation of mass $\sum_z \xi(z) = \sum_z \nu(z)$, we can construct mass transport parameters $\{ g_{i,j} \}$ such that for an $i$ with $\nu(i) > \xi(i)$ and a $j$ with $\nu(j) < \xi(j)$, the quantity $g_{i,j}$ is the fraction of the excess $\nu(i) - \xi(i)$ that goes from $i$ to $j$. In particular, $\{ g_{i,j} \}$ satisfies
\begin{eqnarray}
    g_{i,j} & \in & [0,1] \quad \mbox{for all } i,j \in \ZA \nonumber \\
    \label{eqn:no-mass-flow}
    g_{i,j} & = & 0 \quad \mbox{if } \nu(i) \leq \xi(i) \mbox{ or } \nu(j) \geq \xi(j) \\
    \label{eqn:no-mass-destruction}
    \sum_{j: \nu(j) < \xi(j)} g_{i,j} & = & 1 \quad \mbox{if } \nu(i) > \xi(i)
\end{eqnarray}
and finally
\begin{equation}
  \label{eqn:no-new-mass}
  \sum_{i: \nu(i) > \xi(i)} [\nu(i) - \xi(i)] g_{i,j} = \xi(j) - \nu(j) \quad \mbox{if } \nu(j) < \xi(j).
\end{equation}
Equation (\ref{eqn:no-mass-destruction}) says mass is not destroyed and (\ref{eqn:no-new-mass}) says new mass is not created (all mass entering $j$ must come from $i$'s with excesses).

Define the diagonal elements of the transition rate matrix $L(t)$ to be
\begin{equation}
  \label{eqn:constant-velocity-rate}
  l_{j,j}(t) = \left\{
    \begin{array}{cl}
    \frac{-(\nu(j) - \xi(j))}{T (\mu(t)(j))} & \mbox{if $j$ satisfies } \nu(j) > \xi(j) \\
    0 & \mbox{otherwise}.
    \end{array}
    \right.
\end{equation}
Now define the off-diagonal elements of $L(t)$ to be
\begin{equation}
  \label{eqn:constant-velocity-control}
  l_{j,i}(t) = \left\{
    \begin{array}{cl}
      0 & \mbox{if } \nu(j) \leq \xi(j), i \in \ZA \\
      -l_{j,j}(t) g_{j,i} & \mbox{if } \nu(j) > \xi(j) \mbox{ and } i \neq j.
    \end{array}
  \right.
\end{equation}

We next claim that $\dot{\mu}(t) = L(t)^* \mu(t)$. Indeed, for $i$ such that $\nu(i) \geq \xi(i)$, we
have
\begin{eqnarray*}
  \lefteqn{ \left( L(t)^* \mu(t) \right)(i) } \\*
    & = & \sum_{j} (\mu(t)(j)) ~ l_{j,i}(t) \\
    & = & (\mu(t)(i)) ~ l_{i,i}(t) + \hspace*{-.1in} \sum_{j: j \neq i, \nu(j) \leq \xi(j)} \hspace*{-.1in} (\mu(t)(j))~ l_{j,i}(t) + \hspace*{-.1in} \sum_{j: j \neq i, \nu(j) > \xi(j)} \hspace*{-.1in} (\mu(t)(j))~ l_{j,i}(t) \\
    & \stackrel{\mbox{(a)}}{=} & (\mu(t)(i))~ l_{i,i}(t) + 0 + 0 \\
    & \stackrel{\mbox{(b)}}{=} & \frac{\xi(i) - \nu(i)}{T}.
\end{eqnarray*}
In the above sequence of equalities, the second term in (a) vanished because $\nu(j) \leq \xi(j)$ implies $l_{j,i}(t) = 0$ (see (\ref{eqn:constant-velocity-control})); the third term vanished because, by (\ref{eqn:no-mass-flow}) and noticing that $i$ is the second argument, $\nu(i) \geq \xi(i)$ implies $g_{j,i} = 0$ which in turn implies $l_{j,i}(t) = 0$ again by (\ref{eqn:constant-velocity-control}). Lastly, (b) follows from (\ref{eqn:constant-velocity-rate}).

For $i$ such that $\nu(i) < \xi(i)$, we have
\begin{eqnarray*}
  \lefteqn{ \left( L(t)^* \mu(t) \right)(i) } \\
    & = & \sum_{j} (\mu(t)(j))~ l_{j,i}(t) \\
    & = & (\mu(t)(i))~ l_{i,i}(t) + \hspace*{-.1in} \sum_{j: j \neq i, \nu(j) \leq \xi(j)} \hspace*{-.1in} (\mu(t)(j))~ l_{j,i}(t) + \hspace*{-.1in} \sum_{j: j \neq i, \nu(j) > \xi(j)} \hspace*{-.1in} (\mu(t)(j))~ l_{j,i}(t) \\
    & \stackrel{\mbox{(a)}}{=} & 0 + 0 + \sum_{j: j \neq i, \nu(j) > \xi(j)} (\mu(t)(j))~ (- l_{j,j}(t) g_{j,i}) \\
    & \stackrel{\mbox{(b)}}{=} & \frac{1}{T} \sum_{j: j \neq i, \nu(j) > \xi(j)} (\nu(j) - \xi(j)) g_{j,i} \\
    & \stackrel{\mbox{(c)}}{=} & \frac{\xi(i) - \nu(i)}{T}.
\end{eqnarray*}
In the above sequence of equalities, (a) follows from (\ref{eqn:constant-velocity-rate}) and (\ref{eqn:constant-velocity-control}). Equation (b) follows from (\ref{eqn:constant-velocity-rate}), and (c) follows from (\ref{eqn:no-new-mass}). The above arguments establish $\dot{\mu}(t) = L(t)^* \mu(t)$.

Let us now evaluate the difficulty $S_{[0,T]}(\mu | \nu)$ of passage near this constant velocity path $\mu$. If we show that the integral in the right-hand side of (\ref{eqn:finite-rate-evaluation}) is finite, by Theorem \ref{thm:flow-theorem}, $S_{[0,T]}(\mu | \nu)$ equals this integral. Observe that $l_{i,j}(t)$ is not bounded if one of $\mu(0)(i) = \nu(i)$ or $\mu(T)(i) = \xi(i)$ equals 0, and so we do have some work to do.

The right-hand side of (\ref{eqn:finite-rate-evaluation}) can be expanded to be
\begin{eqnarray}
  \label{eqn:finite-rate-evaluation-detail1}
  \lefteqn{ \int_{[0,T]} \Big[ \sum_{i,j : j \neq i} \Big( (\mu(t)(i))~ l_{i,j}(t) \log \Big( \frac{l_{i,j}(t)}{\lambda_{i,j}(\mu(t))} \Big) } \\
  & & \quad \quad \quad \quad
  - (\mu(t)(i)) l_{i,j}(t) + (\mu(t)(i)) \lambda_{i,j}(\mu(t)) \Big) \Big] ~dt. \nonumber
\end{eqnarray}
From (\ref{eqn:constant-velocity-control}) and (\ref{eqn:constant-velocity-rate}), we get $(\mu(t)(i)) ~l_{i,j}(t) = T^{-1}(\nu(i) - \xi(i)) g_{i,j}$, and this is nonzero only if $\nu(i) > \xi(i)$ and $\nu(j) < \xi(j)$; see (\ref{eqn:no-mass-flow}). For convenience, let us define
\[
  \Upsilon = \{ (i,j) \mid j \neq i, ~\nu(i) > \xi(i), ~\nu(j) < \xi(j) \}.
\]
By assumptions ({\bf A2})-({\bf A3}), $|\log \lambda_{i,j}(\cdot)| \leq |\log C| + |\log c|$. Using these observations, (\ref{eqn:finite-rate-evaluation-detail1}) is upper bounded by
\begin{eqnarray}
  \label{eqn:const-velocity-control-ub-terms}
  \lefteqn{ \int_{[0,T]} \Big[ \sum_{(i,j) \in \Upsilon } \Big(
    T^{-1}(\nu(i) - \xi(i)) g_{i,j} \log \left( \frac{(\nu(i) - \xi(i)) g_{i,j}}{T (\mu(t)(i))} \right) }  \\
    & & \quad + ~ T^{-1} (\nu(i) - \xi(i)) g_{i,j} (|\log C| + |\log c| + 1) + C \Big) \Big] ~dt \nonumber \\
    & \leq & \sum_{(i,j) \in \Upsilon} (\nu(i) - \xi(i)) g_{i,j} |\log ((\nu(i) - \xi(i)) g_{i,j})| \nonumber \\
    & & \quad - ~ \sum_{i: \nu(i) > \xi(i)} (\nu(i) - \xi(i)) T^{-1} \int_{[0,T]} \log (\mu(t)(i)) ~dt \nonumber \\
    & & \quad + ~ \quad || \nu - \xi ||_1 |\log T| \nonumber \\
    & & \quad + ~ \quad || \nu - \xi ||_1 (|\log C| + |\log c| + 1) + C T r^2, \nonumber
\end{eqnarray}
where in arriving at the last three terms we have repeatedly used (\ref{eqn:no-mass-destruction}). The quantity $|| \nu - \xi ||_1$ is the total variation distance between $\xi$ and $\nu$. Let us now bound the first two terms on the right-hand side of (\ref{eqn:const-velocity-control-ub-terms}).

Observing that there is a constant $K$ such that $\sup_{x \in [0,1]} x |\log x| \leq K < +\infty$, the first term on the right-hand side of (\ref{eqn:const-velocity-control-ub-terms}) can be upper bounded as
\begin{eqnarray}
\lefteqn{ \sum_{(i,j) \in \Upsilon} (\nu(i) - \xi(i)) g_{i,j} |\log ((\nu(i) - \xi(i)) g_{i,j})| }  \nonumber \\
& \leq &  \sum_{i:\nu(i) > \xi(i)} (\nu(i) - \xi(i)) |\log (\nu(i) - \xi(i))|
\Biggl( \sum_{j : (i,j) \in \Upsilon} g_{i,j} \Biggr) \nonumber \\
& & \quad + ~ \sum_{(i,j) \in \Upsilon} (\nu(i) - \xi(i)) g_{i,j} |\log g_{i,j}| \nonumber \\
& \leq & \sum_{i:\nu(i) > \xi(i)} (\nu(i) - \xi(i)) |\log (\nu(i) - \xi(i))| + K || \nu - \xi ||_1 \nonumber \\
\label{eqn:const-velocity-control-ub-term1}
& \leq & \sum_i \left| \left( |\nu(i) - \xi(i)| \log |\nu(i) - \xi(i)| \right) \right| + K || \nu - \xi ||_1.
\end{eqnarray}

To bound the second term on the right-hand side of (\ref{eqn:const-velocity-control-ub-terms}), use (\ref{eqn:const-velocity-path}) and employ the change of variable $u = \mu(t)(i)$ to
get\vadjust{\eject}
\begin{eqnarray*}
  \lefteqn{ -(\nu(i) - \xi(i)) T^{-1} \int_{[0,T]} \log (\mu(t)(i)) ~dt = \int_{\nu(i)}^{\xi(i)} \log u ~ du } \\
  & = &  [u \log u - u]_{\nu(i)}^{\xi(i)} \\
  & \leq & |\xi(i) \log \xi(i) - \nu(i) \log \nu(i)| + |\nu(i) - \xi(i)|.
\end{eqnarray*}
Summing this over all $i$, we see that the second term in (\ref{eqn:const-velocity-control-ub-terms}) is upper bounded by
\begin{equation}
  \label{eqn:const-velocity-control-ub-term2}
  \sum_i \left| \xi(i) \log \xi(i) - \nu(i) \log \nu(i) \right| + || \nu - \xi ||_1.
\end{equation}
Since $\MA_1(\ZA)$ is compact, all terms in the upper bounds (\ref{eqn:const-velocity-control-ub-term1}) and (\ref{eqn:const-velocity-control-ub-term2}) are bounded. Substituting (\ref{eqn:const-velocity-control-ub-term1}) and (\ref{eqn:const-velocity-control-ub-term2}) on the right-hand side of  (\ref{eqn:const-velocity-control-ub-terms}), and noticing that $T>0$, we see that the right-hand side of (\ref{eqn:const-velocity-control-ub-terms}) is upper bounded, and this upper bound serves as an upper bound on $S_{[0,T]}(\mu | \nu)$, which we summarize as
\begin{eqnarray}
  S_{[0,T]}(\mu | \nu) & \leq & \sum_i \left| \left( |\nu(i) - \xi(i)| \log |\nu(i) - \xi(i)| \right) \right| \nonumber \\
  & & \quad + ~ \sum_i \left| \xi(i) \log \xi(i) - \nu(i) \log \nu(i) \right| \nonumber \\
  & & \quad + ~ || \nu - \xi ||_1 (|\log T|) \nonumber \\
  \label{eqn:C_3-T}
  & & \quad + ~ || \nu - \xi ||_1 (|\log C| + |\log c| + K + 2) + C T r^2 \\
  & \leq & C_3(T) \nonumber
\end{eqnarray}
for a suitable constant $C_3(T)$ that is independent of $\nu$ and $\xi$. This concludes the proof of the first bullet for the case when all transitions are allowed.

When only those transitions in the directed edge set $\EA$ can occur, since the Markov chain is irreducible (by assumption ({\bf A1})), there exists a finite sequence of intermediate points through which one can move from $\nu$ to $\xi$ in $m = m(r,\EA) < +\infty$ steps:
\[
  \nu = \nu^{(0)} \ra \nu^{(1)} \ra \cdots \ra \nu^{(m)} = \xi.
\]
Consider now the piecewise linear path that moves from $\nu$ to $\xi$ through the above sequence of points with velocities such that each segment is covered in time $T/m$. Then $S_{[0,T]}(\mu|\nu) \leq C_1(T) = m C_3(T/m)$, and the proof of the first bullet is complete.

$S_T(\xi|\nu) \leq C_1(T)$ follows immediately from (\ref{eqn:S(xi-nu)}), whence the second bullet follows.

To see the third bullet, we use (\ref{eqn:C_3-T}). Since $\MA_1(\ZA)$ is a subset of a finite dimensional space, the topology of weak convergence is the same as the topology induced by the total variation metric. In particular, if $\rho_0(\nu,\xi) \ra 0$ then $\nu(i) \ra \xi(i)$ for every $i \in \ZA$. As a consequence, for every $\varepsilon > 0$ and with $T = \varepsilon$, we can choose a $\delta > 0$ such that each of the first four terms in (\ref{eqn:C_3-T}) is upper bounded by $\varepsilon$, and so $\rho_0(\nu,\xi) < \delta$ implies
\[
  S_{\varepsilon}(\xi | \nu) \leq 4 \varepsilon + Cr^2\varepsilon \leq C_2 \varepsilon
\]
for some $C_2 < +\infty$, and the proof of the third bullet and the Lemma is complete. $\hfill \Box$

\section{Proofs of Lemma \ref{lem:uniform-continuity} and Lemma \ref{lem:V-uniform-continuity}}
\label{sec:uniform-continuity}

We begin with two useful lemmas.

\begin{lemma}
  \label{lem:tau-star-linear-bound}
  Let $L(t)$ be a matrix of rates such that the solution $\mu : [0,T] \ra \MA_1(\ZA)$ to the ODE $\dot{\mu}(t) = L(t)^* \mu(t)$ with $\mu(0) = \nu$ has $S_{[0,T]}(\mu | \nu) < +\infty$. There exists a constant $K < +\infty$ such that
  \[
    \int_{[0,T]} \Big[ \sum_{(i,j) \in \EA} (\mu(t)(i)) ~l_{i,j}(t) \Big] ~dt \leq S_{[0,T]}(\mu | \nu) + KT.
  \]
\end{lemma}

\begin{proof}
It is easy to verify that $\tau^*(u-1) = u \log u - u + 1 \geq u - e + 1$ for all $u \geq 0$. By Theorem \ref{thm:flow-theorem}, $S_{[0,T]}(\mu | \nu) < +\infty$ implies that its evaluation is given by (\ref{eqn:finite-rate-evaluation}). Using these two facts, we get
\begin{eqnarray*}
 S_{[0,T]}(\mu | \nu) & = & \int_{[0,T]} \Big[\sum_{(i,j) \in \EA} (\mu(t)(i))~ \lambda_{i,j}(\mu(t)) ~\tau^* \left( \frac{l_{i,j}(t)}{\lambda_{i,j}(\mu(t))} - 1 \right) \Big] ~dt \\
 & \geq & \int_{[0,T]} \Big[\sum_{(i,j) \in \EA} (\mu(t)(i))~ \lambda_{i,j}(\mu(t)) \left( \frac{l_{i,j}(t)}{\lambda_{i,j}(\mu(t))} - e + 1 \right) \Big] ~dt \\
 & \geq & \int_{[0,T]} \Big[\sum_{(i,j) \in \EA} (\mu(t)(i))~ l_{i,j}(t) \Big] ~dt - (e-1)CrT,
\end{eqnarray*}
and the lemma follows.
\end{proof}

The next lemma bounds the increase in the cost due to time scaling on a fixed path between two points.

\begin{lemma}
  \label{lem:speed-up}
  Let $L(t)$ be a matrix of rates such that the solution $\mu : [0,T] \ra \MA_1(\ZA)$ to the ODE $\dot{\mu}(t) = L(t)^* \mu(t)$ with $\mu(0) = \nu$ has $S_{[0,T]}(\mu | \nu) < +\infty$ and $\mu(T) = \xi$. Let $0 < \alpha < +\infty$ be a time scaling. With $T' = T/\alpha$, consider the path $\{ \tilde{\mu}(t) = \mu(\alpha t) \mid t \in [0,T'] \}$ having $\tilde{\mu}(0) = \nu$ and $\tilde{\mu}(T') = \xi$. Then
  \[
    \dot{\tilde{\mu}}(t) = \tilde{L}(t)^* \tilde{\mu}(t), \quad t \in [0,T']
  \]
  where $\tilde{L}(t) = \alpha L(\alpha t)$. Furthermore, the scaled path $\tilde{\mu}: [0,T'] \ra \MA_1(\ZA)$ satisfies
  \begin{eqnarray}
    \label{eqn:push-faster-cost}
    \qquad S_{[0,T']}(\tilde{\mu} | \nu) & \leq & S_{[0,T]}(\mu | \nu) + |\log \alpha| \int_{[0,T]} \Big[ \sum_{(i,j) \in \EA} (\mu(t)(i)) l_{i,j}(t) \Big] ~dt~~ \\
    & &  + ~ \frac{|1-\alpha|}{\alpha} CrT. \nonumber
  \end{eqnarray}
\end{lemma}

\begin{proof}
Clearly $\tilde{\mu}(0) = \mu(0) = \nu$ and $\tilde{\mu}(T') = \mu(\alpha T') = \mu(T) = \xi$. Since $\dot{\mu}(t) = L(t)^* \mu(t)$, we also have
\[
  \dot{\tilde{\mu}}(t) = \frac{d\mu(\alpha t)}{dt} = \alpha \dot{\mu}(\alpha t) = \alpha L(\alpha t)^* \mu(\alpha t) = \alpha L(\alpha t)^* \tilde{\mu}(t)
\]
from which $\tilde{L}(t) = \alpha L(\alpha t)$ is obvious. Its $(i,j)$th entry $\tilde{l}_{i,j}(t)$ equals $\alpha l_{i,j}(\alpha t)$. The cost of $\tilde{\mu} : [0,T'] \ra \MA_1(\ZA)$ is then
\begin{eqnarray*}
  S_{[0,T']}(\tilde{\mu} | \nu) & = & \hspace*{-.1in} \int_{[0,T']} \Big[ \sum_{(i,j) \in \EA} (\tilde{\mu}(t)(i))~ \lambda_{i,j}(\tilde{\mu}(t)) ~\tau^* \left( \frac{\tilde{l}_{i,j}(t)}{\lambda_{i,j}(\tilde{\mu}(t))} - 1 \right) \Big] dt \\
  & = & \hspace*{-.1in} \int_{[0,T']} \Big[ \sum_{(i,j) \in \EA} (\mu( \alpha t)(i))~ \lambda_{i,j}(\mu(\alpha t))~ \tau^* \left( \frac{\alpha l_{i,j}(\alpha t)}{\lambda_{i,j}(\mu(\alpha t))} - 1 \right) \Big] dt \\
  & = & \hspace*{-.1in} \int_{[0,T']} \Big[ \sum_{(i,j) \in \EA} (\mu( \alpha t)(i))~ \lambda_{i,j}(\mu(\alpha t)) \\
  & & \hspace*{-.1in} \times \alpha \Big\{ \tau^* \left( \frac{l_{i,j}(\alpha t)}{\lambda_{i,j}(\mu(\alpha t))} - 1 \right) + \frac{l_{i,j}(\alpha t)}{\lambda_{i,j}(\mu(\alpha t))} (\log \alpha) + \frac{1-\alpha}{\alpha} \Big\} \Big] dt
\end{eqnarray*}
where we have used the fact that
\[
  \tau^*(\alpha u - 1) = \alpha \left\{ \tau^*(u - 1) + u (\log \alpha) + \frac{1-\alpha}{\alpha} \right\}, \quad u \geq 0.
\]
Changing variables from $\alpha t$ to $t$ and continuing, we get
\begin{eqnarray*}
S_{[0,T']}(\tilde{\mu} | \nu) & = & \int_{[0,T]} \Big[ \sum_{(i,j) \in \EA} (\mu(t)(i))~ \lambda_{i,j}(\mu(t)) ~\tau^* \left( \frac{l_{i,j}(t)}{\lambda_{i,j}(\mu(t))} - 1 \right) \Big] ~dt \\
  & & + ~ (\log \alpha) \int_{[0,T]} \Big[ \sum_{(i,j) \in \EA} (\mu(t)(i))~ l_{i,j}(t) \Big] ~dt \\
  & & + ~ \frac{1-\alpha}{\alpha} \int_{[0,T]} \Big[ \sum_{(i,j) \in \EA} (\mu(t)(i))~ \lambda_{i,j}(\mu(t)) \Big] ~dt.
\end{eqnarray*}
Since the first term on the right-hand side above is $S_{[0,T]}(\mu | \nu)$ and $\lambda_{i,j}(\cdot) \leq C$, (\ref{eqn:push-faster-cost}) follows.
\end{proof}

\begin{proof}[Proof of Lemma \ref{lem:uniform-continuity}]
Fix $T > 0$. Fix an arbitrary $\varepsilon$ such that $0 < \varepsilon < T/4$. Let $\delta > 0$ be as given by part 3 of Lemma \ref{lem:S_T-bounding} so that $\rho_0(\nu,\xi) < \delta$ implies $S_{\varepsilon}(\xi|\nu) \leq C_2 \varepsilon$.

Let $\{ (\nu_i, \xi_i), i = 1,2 \}$ be two points in $\ZA \times \ZA$. By an abuse of notation, let $\rho_T$ given by
\[
  \rho_T((\nu_1, \xi_1), (\nu_2, \xi_2)) = \max \{ \rho_0(\nu_1, \nu_2), \rho_0 (\xi_1, \xi_2) \}.
\]
denote the metric on $\ZA \times \ZA$. Let $\rho_T((\nu_1, \xi_1), (\nu_2, \xi_2)) < \delta$. We need to show that $S_T(\xi_1 | \nu_1)$ and $S_T(\xi_2 | \nu_2)$ are close to each other.

Obviously, $\rho_0(\nu_1, \nu_2) < \delta$ and $\rho_0(\xi_2, \xi_1) < \delta$. Let $\mu$ denote the minimum cost path from $\nu_2$ to $\xi_2$ in time $T$ with cost $S_T(\xi_2 | \nu_2)$. Consider the path from $\nu_1$ to $\xi_1$ as follows:
\begin{itemize}
  \item Traverse the path from $\nu_1$ to $\nu_2$ in time $[0, \varepsilon]$, as given by part 3 of Lemma \ref{lem:S_T-bounding}. This traversal costs  at most $C_2 \varepsilon$.
  \item Given the optimal $[0,T]$-path $\mu$ from $\nu_2$ to $\xi_2$, consider the sped-up path $\tilde{\mu} : [0,T-2\varepsilon] \ra \MA_1(\ZA)$ given by $\tilde{\mu}(t) = \mu(\alpha t)$ with $\alpha = T/(T - 2\varepsilon)$. Travel from $\nu_2$ to $\xi_2$ in the duration $[\varepsilon, T-\varepsilon]$ along the path $\tilde{\mu}$.
  \item Traverse the path from $\xi_2$ to $\xi_1$ in time $[0, \varepsilon]$, again as given by part 3 of Lemma \ref{lem:S_T-bounding}. This traversal's cost  is also at most $C_2 \varepsilon$.
\end{itemize}
The minimum cost for traversal from $\nu_1$ to $\xi_1$ is at most the sum of these paths. Hence, by Lemmas \ref{lem:speed-up} and \ref{lem:tau-star-linear-bound}, we get
\begin{eqnarray*}
  S_T(\xi_1 | \nu_1) \hspace*{-.05in} & \leq & \hspace*{-.05in} 2 C_2 \varepsilon + S_T(\xi_2 | \nu_2) + \Big( \log \frac{T}{T-2\varepsilon} \Big) (S_T(\xi_2 | \nu_2) + K T) + \frac{2 \varepsilon}{T} C r T \\
  & \leq & \hspace*{-.05in} S_T(\xi_2 | \nu_2) + 2 C_2 \varepsilon + \Big( \frac{T}{T-2\varepsilon} - 1 \Big) (C_1(T) + K T) + 2 C r \varepsilon
\end{eqnarray*}
where we used $\log u \leq u - 1$ for $u > 0$. Observing that
\[
  \varepsilon < T/4 \Rightarrow \frac{T}{T - 2\varepsilon} - 1 = \frac{2\varepsilon}{T - 2\varepsilon} \leq \frac{4 \varepsilon}{T},
\]
we deduce that
\[
  S_T(\xi_1 | \nu_1) \leq S_T(\xi_2 | \nu_2) + C_4(T) \varepsilon
\]
where we may take $C_4(T) = 2 C_2 + 4K + 4C_1(T)/T + 2Cr$. Reversing the roles of $(\nu_1, \xi_1)$ and $(\nu_2, \xi_2)$, we deduce
\[
  |S_T(\xi_1 | \nu_1) - S_T(\xi_2 | \nu_2)| \leq C_4(T) \varepsilon.
\]
This concludes the proof that $(\nu,\xi) \mapsto S_T(\xi|\nu)$ is uniformly continuous.
\end{proof}

We now provide the proof of the result on uniform continuity of the analogous quantity $V(\xi | \nu)$.

\begin{proof}[Proof of Lemma \ref{lem:V-uniform-continuity}]
Fix $\varepsilon > 0$ and choose $\delta$ as in the third part of Lemma \ref{lem:S_T-bounding}. Let $(\nu_1, \xi_1)$ and $(\nu_2, \xi_2)$ be such that the starting points are $\delta$-close to each other and so are the ending points, that is, $\rho_0(\nu_1,\nu_2) < \delta$ and $\rho_0(\xi_1,\xi_2) < \delta$. Consider the following path:
\begin{itemize}
  \item Traverse the path from $\nu_1$ to $\nu_2$ in time $[0,\varepsilon]$, as given by part 3 of Lemma \ref{lem:S_T-bounding}. This traversal costs at most $C_2 \varepsilon$.
  \item Traverse by a path from $\nu_2$ to $\xi_2$ in finite time by a path with cost at most $V(\xi_2 | \nu_2) + C_2 \varepsilon$.
  \item Traverse the path from $\xi_2$ to $\xi_1$ in time $[0,\varepsilon]$, again as given by part 3 of Lemma \ref{lem:S_T-bounding}. This traversal's cost is also at most $C_2 \varepsilon$.
\end{itemize}
We then have
\[
  V(\xi_1 | \nu_1) \leq C_2 \varepsilon + ( V ( \xi_2 | \nu_2 ) + C_2 \varepsilon) + C_2 \varepsilon = V(\xi_2 | \nu_2) +  3 C_2 \varepsilon.
\]
Reversing the roles of $(\nu_1, \xi_1)$ and $(\nu_2, \xi_2)$ and via a similar argument, we deduce that
\[
  |V(\xi_1 | \nu_1) - V(\xi_2 | \nu_2)| \leq 3 C_2 \varepsilon
\]
which shows that $(\nu, \xi) \mapsto V(\xi | \nu)$ is uniformly continuous.
\end{proof}

\section{Proof of Theorem \ref{thm:JointLDP}}
\label{sec:ProofJointLDP}

Again, we proceed through a sequence of lemmas. Let $\nu_N \ra \nu$ weakly. By Theorem \ref{thm:flow-theorem}, the sequence of laws of the terminal measure $(p^{(N)}_{\nu_N,T}, N \geq 1)$ satisfies the large deviation principle with speed $N$ and good rate function $S_T(\xi|\nu)$. By Varadhan's lemma, for every $f \in C_b(\MA_1(\ZA))$, we have
\begin{equation}
  \label{eqn:varadhan-lemma-terminaltime}
  \lim_{N \ra +\infty} \frac{1}{N} \log \int_{\MA_1(\ZA)} e^{N f} ~dp^{(N)}_{\nu_N,T} = \sup_{\xi \in \MA_1(\ZA)} [f(\xi) - S_T(\xi|\nu)].
\end{equation}
Let us define
\begin{equation}
  \label{eqn:Lambda}
  \Lambda(f|\nu) = \sup_{\xi \in \MA_1(\ZA)} [f(\xi) - S_T(\xi|\nu)].
\end{equation}
Observe that the rate function admits the characterization (see, e.g., \cite[Th. 4.4.2]{DZ})
\begin{equation}
  \label{eqn:Lambda-to-Rate}
  S_T(\xi|\nu) = \sup_{f \in C_b(\MA_1(\ZA))} [f(\xi) - \Lambda(f|\nu)].
\end{equation}

\begin{lemma}
  \label{lem:Lambda-continuity}
  Let $f \in C_b(\MA_1(\ZA))$. The mapping $\nu \in \MA_1(\ZA) \mapsto \Lambda(f|\nu) \in \mathbb{R}$ is continuous.
\end{lemma}

\begin{proof}
Since $f$ is continuous, by Lemma \ref{lem:uniform-continuity}, the mapping
\[
  \eta: (\nu, \xi) \in \MA_1(\ZA) \times \MA_1(\ZA) \mapsto f(\xi) - S_T(\xi|\nu) \in \mathbb{R}
\]
is jointly continuous. As $\MA_1(\ZA)$ is compact, the supremum in the definition of (\ref{eqn:Lambda}) is attained.

Let $\nu_N \ra \nu$ weakly, and for each $\nu_N$, let $\xi_N$ denote a point where the supremum in the definition of (\ref{eqn:Lambda}) is attained. In other words, $\Lambda(f|\nu_N) = \eta(\nu_N, \xi_N)$ for each $N$. The sequence $((\nu_N, \xi_N), N \geq 1)$ has a convergent subsequence that converges to $(\nu, \xi)$, for some $\xi$. Reindex so that we may take $(\nu_N, \xi_N) \ra (\nu, \xi)$ as $N \ra +\infty$. By the continuity of $\eta$,
\[
  \Lambda(f|\nu_N) = \eta(\nu_N, \xi_N) \ra \eta(\nu, \xi)
\]
as $N \ra +\infty$. The proof will be complete if we can show that $\Lambda(f|\nu) = \eta(\nu, \xi)$, that is, the supremum in $\eta(\nu, \cdot)$ is attained at $\xi$.

To see this, observe that for any $\xi'$, we have $\eta(\nu_N, \xi') \leq \eta(\nu_N, \xi_N)$, and so
\[
  \eta(\nu, \xi') = \lim_{N \ra +\infty} \eta(\nu_N, \xi') \leq \limsup_{N \ra +\infty} \eta(\nu_N, \xi_N) = \eta(\nu, \xi).
\]
This completes the proof of the lemma.
\end{proof}

Our next result show that the convergence in (\ref{eqn:varadhan-lemma-terminaltime}) is uniform. This is where our uniform large deviation result for nonchaotic initial conditions comes in handy.

Recall that $\MA_1^{(N)}(\ZA) \subset \MA_1(\ZA)$ is the subset of values taken by the initial empirical measure $\nu$ when there are $N$ particles.

\begin{lemma}
  \label{lem:uniform-convergence-lv-functional}
  The convergence in (\ref{eqn:varadhan-lemma-terminaltime}) is uniform in the following sense: for each $f \in C_b(\MA_1(\ZA))$, we have
  \begin{equation}
    \label{eqn:uniform-convergence-lv-functional}
    \lim_{N \ra +\infty} \sup_{\nu \in \MA_1^{(N)}(\ZA)}
      \left|
        \frac{1}{N} \log \int_{\MA_1(\ZA)} e^{Nf} ~dp^{(N)}_{\nu, T} - \Lambda(f|\nu)
      \right| = 0.
  \end{equation}
\end{lemma}

\begin{proof}
We will prove this by contradiction. Suppose the above limit is not zero. Then there is an $\varepsilon > 0$ and an infinite subset $\mathbb{V}_1 \subset \mathbb{N}$ such that
\[
  \sup_{\nu \in \MA_1^{(N)}(\ZA)}
      \left|
        \frac{1}{N} \log \int_{\MA_1(\ZA)} e^{Nf} ~dp^{(N)}_{\nu, T} - \Lambda(f|\nu)
      \right| > \varepsilon, \quad \mbox{for every } N \in \mathbb{V}_1,
\]
that is, the violations occur infinitely often. So we can find a sequence $( \nu_{N} )_{N \in \mathbb{V}_1 }$ such that
\[
  \left|
        \frac{1}{N} \log \int_{\MA_1(\ZA)} e^{N f} ~dp^{(N)}_{\nu_{N}, T} - \Lambda(f|\nu_{N})
      \right| > \varepsilon, \quad \mbox{for every } N \in \mathbb{V}_1.
\]
Extract a further subsequence, which is another infinite subset $\mathbb{V}_2 \subset \mathbb{V}_1$, such that $(\nu_{N})_{N \in \mathbb{V}_2} \ra \nu$ for some $\nu$. By Lemma \ref{lem:Lambda-continuity}, $(\Lambda(f|\nu_N))_{N \in \mathbb{V}_2} \ra \Lambda(f|\nu)$, and so

\begin{equation}
  \label{eqn:varadhan-lemma-tt-contradiction}
  \left|
        \frac{1}{N} \log \int_{\MA_1(\ZA)} e^{N f} ~dp^{(N)}_{\nu_{N}, T} - \Lambda(f|\nu)
      \right| > \frac{\varepsilon}{2}, \mbox{ for all sufficiently large } N \in \mathbb{V}_2.
\end{equation}
Construct a new initial state sequence $(\nu_N)_{N \geq 1}$ that matches with the above subsequence for $N \in \mathbb{V}_2$ and such that $\nu_N \ra \nu$. For such a sequence, by Theorem \ref{thm:flow-theorem} and Varadhan's lemma, (\ref{eqn:varadhan-lemma-terminaltime}) holds. But (\ref{eqn:varadhan-lemma-tt-contradiction}) for all sufficiently large $N \in \mathbb{V}_2$ is a contradiction to (\ref{eqn:varadhan-lemma-terminaltime}).
\end{proof}

\begin{proof}[Proof of Theorem \ref{thm:JointLDP}]
Consider the joint measure $\wp^{(N)}_{0,T}$ given by
\[
  d\wp^{(N)}_{0,T}(\nu, \xi) = d\wp^{(N)}_0(\nu) dp^{(N)}_{\nu, T}(\xi).
\]
We shall apply Feng and Kurtz's \cite[Prop. 3.25]{Feng-Kurtz}. To do this, we need to verify three conditions listed below.
\begin{itemize}
  \item {\em Exponential tightness}. The sequence $(\wp^{(N)}_{0,T}, N \geq 1)$, which comprises of probability measures on the compact product space, is trivially exponentially tight.
  \item {\em Uniform convergence of the Laplace-Varadhan functional in the initial condition}. For each $N$, the probability measure $p^{(N)}_{\nu, T}$ is supported on the compact subset $\MA_1^{(N)}(\ZA)$. By Lemma \ref{lem:uniform-convergence-lv-functional}, for each $f \in C_b(\MA_1(\ZA))$, the convergence of the Laplace-Varadhan functional is uniform, as given in (\ref{eqn:uniform-convergence-lv-functional}).
  \item {\em Continuity of Laplace-Varadhan functional in the initial condition}. For each $f \in C_b(\MA_1(\ZA))$, the function $\nu \mapsto \Lambda(f|\nu)$ is continuous, by Lemma \ref{lem:Lambda-continuity}.
\end{itemize}

Under the above conditions, Feng and Kurtz demonstrated in \cite[Prop. 3.25 and Rem. 3.26]{Feng-Kurtz} that if the first marginal sequence $(\wp^{(N)}_0, N \geq 1)$ satisfies the large deviation principle with speed $N$ and good rate function $s$, then so does the sequence of joint laws $(\wp^{(N)}_{0,T}, N \geq 1)$ with good rate function $S_{0,T}(\nu,\xi) = s(\nu) + S_T(\xi|\nu)$. This concludes the proof.
\end{proof}


\begin{appendix}

\section{Multiple $\omega$-limit sets}

In this appendix, we verify that the classical program of Freidlin-Wentzell \cite[Ch. 6]{Freidlin} can be extended to our setting. There are primarily two things to keep in mind. First, for all finite $N$, we have a jump process on the simplex. Second, the quantity $V(\cdot | \cdot)$ defined in (\ref{eqn:quasipotential}) is only uniformly continuous and not Lipschitz continuous. But this uniform continuity suffices. Though the changes are minor, we provide the entire sequence of lemmas with modified proofs for completeness and ease of verification.

\subsection{Auxiliary results}

We begin with a subset of the auxiliary results in \cite[Ch. 6]{Freidlin} that were shown for diffusions on a compact manifold.

\begin{lemma}{\em (Freidlin and Wentzell {\cite[Ch.6, Lemma
1.2]{Freidlin}})}
\label{lem:FW-1.2}
For any $\varepsilon > 0$ and any compact set $K \subset \MA_1(\ZA)$, there exists a $T_0$ such that for any $\nu, \xi \in K$ there exists a function $\mu(t), t \in [0,T], \mu(0) = \nu, \mu(T) = \xi, T \leq T_0$ with $S_{[0,T]}(\mu | \nu) \leq V(\xi | \nu) + \varepsilon$.
\end{lemma}

\begin{proof}
Fix $\varepsilon > 0$. By part 3 of Lemma \ref{lem:S_T-bounding}, with the constant $C_2$ as in that Lemma, and $\varepsilon_1 = \varepsilon/(4C_2)$, there is a $\delta_1 \in (0,\varepsilon_1)$ such that two points within a distance $\delta_1$ can be connected by a path of duration $\varepsilon_1$ and cost at most $C_2 \varepsilon_1 = \varepsilon/4$.

By Lemma \ref{lem:V-uniform-continuity}, there is a $\delta_2$ such that
\[
  \max \{\rho_0(\xi_1, \xi_2), \rho_0(\nu_1, \nu_2) \} < \delta_2 \mbox{ implies } | V(\xi_2 | \nu_2) - V(\xi_1 | \nu_1) | \leq \varepsilon/4.
\]

Let $\delta = \min \{ \delta_1, \delta_2 \}$. Choose a finite $\delta$-net $\{ \nu_i \}$ of points in $K$. Connect them with curves
$\mu_{i,j}(t), t \in [0, T_{i,j}], \mu_{i,j}(0) = \nu_i, \mu_{i,j}(T_{i,j}) = \nu_j$
such that
\[
  S_{[0,T_{i,j}]}(\mu_{i,j} | \nu_i) \leq V(\nu_j | \nu_i) + \varepsilon/4.
\]
For arbitrary $\nu, \xi \in K$, let $\nu_k$ and $\nu_l$ be the respective closest points on the net. We can now find a path from $\nu$ to $\nu_k$, then to $\nu_l$ along $\mu_{k,l}$, and then to $\xi$, with overall cost at most
\[
  \varepsilon/4 + (V(\nu_l|\nu_k) + \varepsilon/4) + \varepsilon/4 =  V(\nu_l|\nu_k) + 3 \varepsilon/4 \leq V(\xi|\nu) + \varepsilon
\]
where the last inequality follows from the uniform continuity in Lemma \ref{lem:V-uniform-continuity} and the choice of $\delta$. The duration of the path is $T_{k,l} + 2 \varepsilon_1 \leq T_0 := (\max_{i,j} T_{i,j}) + 2 \varepsilon_1$.
\end{proof}

For a set $A \subset \MA_1(\ZA)$, let $[A]_{\delta}$ denote the (open) $\delta$-neighborhood of $A$. Its closure will be denoted $\overline{[A]_{\delta}}$. For the following lemma, recall the notion of equivalence between two points: $\nu \sim \xi$ if $V(\xi|\nu) = V(\nu | \xi) = 0$ (see Section~\ref{subsec:Inv-measure-multiple-eq}).

\begin{lemma}{\em (Freidlin and Wentzell \cite[Ch.6, Lemma 1.6]{Freidlin})}
\label{lem:FW-1.6}
Let all points of a compact set $K$ be equivalent to each other, but not to any other point in $\MA_1(\ZA)$. For any $\varepsilon > 0$, $\delta > 0$, $\nu, \xi \in K$, there exists a $T > 0$ and a function $\mu(t), 0 \leq t \leq T$ with $\mu(0) = \nu, \mu(T) = \xi, \mu(t) \in [K]_{\delta}$ for all $t \in [0,T]$, and $S_{[0,T]}(\mu | \nu) < \varepsilon$.
\end{lemma}

\begin{proof}
Fix $\varepsilon > 0, \delta > 0, \nu, \xi \in K$. We can find a sequence $(T_n, n \geq 1)$ and paths $\mu^{(n)}:[0,T_n] \ra \MA_1(\ZA)$ such that
$\mu^{(n)}(0) = \nu, \mu^{(n)}(T_n) = \xi$ for all $n \geq 1$, and $\varepsilon > S_{[0,T_n]} (\mu^{(n)} | \nu) \ra 0$ as $n \ra +\infty$. Observe that $\MA_1(\ZA) \setminus [K]_{\delta}$ is compact, and so if an infinite number of $\mu^{(n)}$ left $[K]_{\delta}$, there is a limit point $z$ outside $[K]_{\delta}$. Using part 3 of Lemma \ref{lem:S_T-bounding}, and $S_{[0,T_n]} (\mu^{(n)} | \nu) \ra 0$ as $n \ra +\infty$, it follows that $V(z | \nu) = V(\xi | z) = 0$. Together with $V(\nu | \xi) = 0$, we conclude that $V(\nu | z) = 0$ and so $z \sim \nu$. But then $z$ is an equivalent point outside $K$, which is a contradiction.

Hence $\mu^{(n)}$ goes outside $[K]_{\delta}$ for finitely many $n$. Take the first index larger than these. The corresponding path remains completely within $[K]_{\delta}$, and meets all the other requirements.
\end{proof}

In this section, we shall use the notation
$$\tau_A := \inf \{ t \geq 0 ~|~ \mu_N(t) \notin A \}.$$
The law for this exit time depends on $N$ and $\nu$ through the law $p^{(N)}_{\nu}$ for $\mu_N$. This dependence will be assumed as understood and will be suppressed for brevity.

\begin{lemma}{\em (Freidlin and Wentzell \cite[Ch.6, Lemma 1.7]{Freidlin})}.
Let all points of a compact set $K$ be equivalent to each other and let $K \neq \MA_1(\ZA)$. For a $\delta > 0$, let
$$\tau_{[K]_{\delta}} := \inf \{ t \geq 0 ~|~ \mu_N(t) \notin [K]_{\delta} \}.$$
For any $\varepsilon > 0$, there exists a $\delta > 0$ such that for all sufficiently large $N$ and all $\nu \in [K]_{\delta}$, we have
\[
  \mathbb{E} [\tau_{[K]_{\delta}}] < e^{+N\varepsilon}
\]
where the expectation is with respect to the measure $p^{(N)}_{\nu}$.
\end{lemma}

\begin{proof}
Fix $\varepsilon > 0$. Again by part 3 of Lemma \ref{lem:S_T-bounding}, there is a $\delta_1 > 0$ such that two points $\delta_1$-close have a path connecting them of duration $\varepsilon_1 = \varepsilon/(4C_2)$ and cost at most $\varepsilon/4$. Choose $\xi$ outside $K$ such that $\rho_0(\xi, K) < \delta_1$. Choose $\delta < \rho_0 (\xi, K) / 2$; we thus have $0 < \delta < \rho_0 (\xi, K) / 2 < \rho_0 (\xi, K) < \delta_1$.

Consider a finite $\delta$-net of $K$. Lemma \ref{lem:FW-1.2} assures existence of paths that connect any pair of the net with cost at most $\varepsilon/4$. Let $T_0'$ denote the maximum time among these paths, where the maximum is over pairs belonging to the net, and let $T_0 = T_0' + 2 \varepsilon_1$. Traverse from any $\nu \in [K]_{\delta}$ to its nearest point on the net, then traverse from that point to the point on the net nearest to $\xi$, and thence to $\xi$. Now extend this path following the McKean-Vlasov dynamics so that the total duration is now $T_0$. This last appendage incurs no additional cost. Denote by $\mu$ the resulting path of duration $T_0$. Clearly $S_{[0,T_0]}(\mu | \nu) \leq 3 \varepsilon/4$.

Now, any trajectory that is strictly $\delta$-close to the trajectory $\mu$ exits $[K]_{\delta}$ at least once in the interval $[0,T_0]$ because for some $t \in [0,T_0]$, we have $\mu(t) = \xi$ which is at a distance greater than $2 \delta$ from $K$. We then have
\begin{eqnarray*}
  p^{(N)}_{\nu} \left\{ \tau_{[K]_{\delta}} < T_0 \right\} & \geq & p^{(N)}_{\nu} \left\{ \rho_{T_0} ( \mu_N, \mu ) < \delta \right\} \\
  & \geq & e^{-3\varepsilon N / 4}, \mbox{ for all } \nu \in [K]_{\delta}, \mbox{ for all } N \geq \mbox{ some } N_0,
\end{eqnarray*}
where the last inequality holds by (\ref{eqn:uniform-ldp-lb}) in Corollary \ref{cor:uniform-ldp-flows}. Consequently
\[
  p^{(N)}_{\nu} \left\{ \tau_{[K]_{\delta}} \geq T_0 \right\} \leq 1 - e^{-3\varepsilon N / 4}, \quad \mbox{ for all } \nu \in [K]_{\delta}, N \geq N_0.
\]
This uniform bound, the Markov property, and induction imply
\[
  p^{(N)}_{\nu} \left\{ \tau_{[K]_{\delta}} \geq mT_0 \right\} \leq \left( 1 - e^{-3\varepsilon N / 4} \right)^m,
\]
from which we obtain
\[
  \mathbb{E} [\tau_{[K]_{\delta}}] \leq T_0 \sum_{m \geq 0} (1 - e^{-3 \varepsilon N / 4})^m = T_0 e^{3 \varepsilon N / 4} < e^{\varepsilon N}
\]
where the last inequality holds for all sufficiently large $N$. This concludes the proof.
\end{proof}

\begin{lemma}{\em (Freidlin and Wentzell \cite[Ch.6, Lemma 1.8]{Freidlin})}.
\label{lem:FW-1.8}
Let $K$ be an arbitrary compact subset of $\MA_1(\ZA)$ and let $G$ be a neighborhood of $K$. For any $\varepsilon > 0$, there exists a $\delta > 0$ such that for all sufficiently large $N$ and all $\nu$ belonging to $\overline{g}$, with $g = [K]_{\delta}$ and $\overline{g} = \overline{[K]_{\delta}}$, we have
\[
  \mathbb{E} \left[\int_{[0,\tau_G]} {\bf 1}_{g} \left( \mu_N(t) \right)~ dt \right] > e^{-\varepsilon N},
\]
where the expectation is with respect to the measure $p^{(N)}_{\nu}$.
\end{lemma}

\begin{proof}
Fix $\varepsilon > 0$. Choose $\delta_1$ small enough so that $[K]_{\delta_1} \subset G$. Next, choose $\delta_2$ as in part 3 of Lemma \ref{lem:S_T-bounding} so that with $\varepsilon_1 = \varepsilon/(2 C_2)$, any two $\delta_2$-close points can be connected by a path of duration $\varepsilon_1$ and cost at most $\varepsilon/2$. Now let $\delta < \min \{ \delta_1, \delta_2/2 \}$ and set $g = [K]_{\delta}$.

Fix $T > \varepsilon_1$. Take any $\nu \in \overline{g}$. Connect it to the closest point on $K$ (via a path of duration $\varepsilon_1$ and cost $\leq \varepsilon/2$) and then extend via the McKean-Vlasov path with this initial condition for a further duration of $T - \varepsilon_1$. Call the entire path of duration $T$ as $\mu$. So long as $\mu$ is inside the $\delta/3$ neighborhood of $K$, any $\delta/3$ neighborhood of $\mu$ lies completely inside $g = [K]_{\delta}$. By assumptions ({\bf A2})-({\bf A3}), the McKean-Vlasov dynamics has a bounded velocity field. Consequently, the part of $\mu$ that begins at $K$ and until either its exit from $\overline{[K]_{\delta/3}}$ or time $T$, whichever occurs earlier, is of duration at least $t_0$ for some $t_0 > 0$, independent of the starting point. It follows that
\[
  \{ \rho_T (\mu_N, \mu) \leq \delta/3 \} \mbox{ implies } \{ \tau_g \geq \min \{T,t_0\} \}.
\]
Furthermore, $\tau_G \geq \tau_g$ and ${\bf 1}_g(\mu_N(t)) = 1$ until the random path exits $g$. Thus
\begin{eqnarray*}
  \mathbb{E} \left[ \int_{[0,\tau_G]} {\bf 1}_g(\mu_N(t)) ~dt \right]
    & \geq & \mathbb{E} [\tau_g] \\
    & \geq & \mathbb{E} [\tau_g \cdot {\bf 1}\{ \rho_T (\mu_N, \mu) \leq \delta/3 \}] \\
    & \geq & \min \{ T, t_0 \}  \cdot p^{(N)}_{\nu} \left\{ \rho_T (\mu_N, \mu) \leq \delta/3 \right\} \\
    & \geq & \min \{ T, t_0 \}  \cdot e^{- N \varepsilon/2} \quad (\mbox{by (\ref{eqn:uniform-ldp-lb})})\\
    & \geq & e^{-N \varepsilon},
\end{eqnarray*}
where the last two inequalities hold for all sufficiently large $N$.
\end{proof}

\begin{lemma}{\em (Freidlin and Wentzell \cite[Ch.6, Lemma 1.9]{Freidlin})}.
\label{lem:FW-1.9}
Let $K$ be a compact subset of $\MA_1(\ZA)$ not containing any $\omega$-limit set entirely. There exist positive constants $c$ and $T_0$ such that for all sufficiently large $N$, any $T>T_0$, and any $\nu \in K$, we have
\[
  p^{(N)}_{\nu} \{ \tau_K > T \} \leq e^{-N c (T-T_0)}.
\]
\end{lemma}

\begin{proof}
For a sufficiently small $\delta$, the closed $\delta$-neighborhood of $K$, denoted $\overline{[K]_{\delta}}$, does not contain any $\omega$-limit set entirely. Indeed, if this were not true, we can find a sequence of $\delta \downarrow 0$ such that each set in the nested decreasing sequence of sets $\overline{[K]_{\delta}}$ contains an $\omega$-limit set entirely. For a $\delta$, define $\Omega(\delta)$ to be the closure of the union of $\omega$-limit sets contained in $\overline{[K]_{\delta}}$. Clearly, $\Omega(\delta)$ is a positively invariant set, and the family indexed by $\delta$ is a nested decreasing sequence of nonempty compact sets. Then $\cap_{\delta} \Omega(\delta)$ is a nonempty compact invariant set in $K$. Further, it contains an $\omega$-limit set entirely, a contradiction.

For $\nu \in \overline{[K]_{\delta}}$, denote by $\tau(\nu)$ the time for first exit of the solution to the McKean-Vlasov equation with initial condition $\nu$ from the set $\overline{[K]_{\delta}}$. Since $\overline{[K]_{\delta}}$ does not contain any $\omega$-limit set entirely, $\tau(\nu) < +\infty$ for all $\nu \in \overline{[K]_{\delta}}$. The function $\tau(\nu)$ is upper semicontinuous, and consequently, it attains its largest value $\max_{\nu \in \overline{[K]_{\delta}}} \tau(\nu) = T_1 < +\infty$.

Set $T_0 = T_1 + 1$ and consider all paths of duration $T_0$ that take values only in $\overline{[K]_{\delta}}$. It is easy to see that this set is closed. It follows that for each $\nu \in \overline{[K]_{\delta}}$, we have that $S_{[0,T_0]}(\cdot | \nu)$ attains its minimum $A(\nu)$ on this set. Further, the mapping $\nu \mapsto A(\nu)$ is continuous, as can be shown by an easy application of Lemmas \ref{lem:S_T-bounding}, \ref{lem:tau-star-linear-bound} and \ref{lem:speed-up}. Thus $A := \min_{\nu \in \overline{[K]_{\delta}}} A(\nu)$ is attained. This minimum is strictly positive since there are no trajectories of the McKean-Vlasov equation among the paths under consideration.

Fix $\varepsilon < A$,  a $\nu \in K$, and consider the family of paths
\[
  \Phi_{\nu}(A - \varepsilon/2) = \{ \mu:[0,T_0] \ra \MA_1(\ZA) ~|~ S_{[0,T_0]}(\mu | \nu) \leq A - \varepsilon/2\}.
\]
Any path in this set exits $\overline{[K]_{\delta}}$ in the interval $[0,T_0]$. With initial state $\nu$, the event $\{ \tau_K > T_0 \}$ implies that the trajectory remains entirely within $K$, and since any path in $\Phi_{\nu}(A - \varepsilon/2)$ exits $\overline{[K]_{\delta}}$, we must have
$$\rho_{T_0} (\mu_N, \Phi_{\nu}(A - \varepsilon/2)) \geq \delta.$$
It follows that
\begin{eqnarray*}
  p^{(N)}_{\nu} \{ \tau_K > T_0 \} & \leq & p^{(N)}_{\nu} \{ \rho_{T_0} (\mu_N, \Phi_{\nu}(A - \varepsilon/2)) \geq \delta \}.
\end{eqnarray*}
By considering any $\nu_1 \in K$, and by using (\ref{eqn:uniform-ldp-ub}) of Corollary \ref{cor:uniform-ldp-flows}, we have
\begin{eqnarray*}
    \sup_{\nu_1 \in K} p^{(N)}_{\nu_1} \{ \tau_K > T_0 \} & \!\!\leq \!\!& \sup_{\nu_1 \in K} p^{(N)}_{\nu_1} \{ \rho_{T_0} (\mu_N, \Phi_{\nu}(A - \varepsilon/2)) \geq \delta \} \\
    & \!\!\leq\!\! & \exp \{-N (\inf \{ S_{[0,T]}(\mu | \nu_1) | \mu \,{\in}\, \Phi_{\nu_1}(A \,{-}\, \varepsilon/2) \}
                     \,{-}\, \varepsilon/2) \} \\
    &\!\!\!\! & \quad \quad \quad \quad \mbox{ for all sufficiently large } N \\
    & \!\!\leq\!\! & e^{-N (A - \varepsilon)} \quad \mbox{ for all sufficiently large } N.
\end{eqnarray*}
For our fixed $\nu \in K$, the Markov property then implies
\begin{eqnarray*}
  p^{(N)}_{\nu} \left\{ \tau_K > (m+1)T_0 \right\} & \leq &
  \mathbb{E} [{\bf 1}\{\tau_K \,{>}\, mT_0\} \cdot \mathbb{E} [ {\bf 1}\{ \tau_K \,{>}\, T_0\} ~|~ \mu^{(N)}(mT_0) ] ] \\
  & \leq & p^{(N)}_{\nu} \left\{ \tau_K > mT_0 \right\} \cdot \left( \sup_{\nu_1 \in K} p^{(N)}_{\nu_1}
  \left\{ \tau_K > T_0 \right\} \right) \\*
  & \leq & p^{(N)}_{\nu} \left\{ \tau_K > mT_0 \right\} \cdot e^{-N(A - \varepsilon)}.
\end{eqnarray*}
By induction, for a $T > T_0$, we have
\begin{eqnarray*}
  p^{(N)}_{\nu} \left\{ \tau_K > T \right\} & \leq & p^{(N)}_{\nu} \left\{ \tau_K > \left\lfloor \frac{T}{T_0} \right\rfloor T_0 \right\} \\
  & \leq & e^{-N(A-\varepsilon) \left( \left\lfloor \frac{T}{T_0} \right\rfloor \right) } \\
  & \leq & e^{-N(A - \varepsilon)(T/T_0-1)} = e^{-Nc(T-T_0)}
\end{eqnarray*}
for $c = (A - \varepsilon)/T_0$, and this completes the proof.
\end{proof}

The above theorem has the following immediate corollary.

\begin{corollary}{\em (Freidlin and Wentzell \cite[Ch.6, Corollary to Lemma 1.9]{Freidlin})}.
Let $K$ be a compact set not containing any $\omega$-limit set entirely. There exists a positive integer $N_0$ and a positive constant $c$ such that for $N \geq N_0$ and any $\nu \in K$, we have
\[
  \mathbb{E} [\tau_K] \leq T_0 + 1/(cN_0)
\]
where the expectation is with respect to the measure $p^{(N)}_{\nu}$.
\end{corollary}

Recall the definition of $V$ given in (\ref{eqn:quasipotential}), and the notion of equivalence on $\MA_1(\ZA)$ given in Section \ref{subsec:Inv-measure-multiple-eq}. Under condition ({\bf B}) in Section \ref{subsec:Inv-measure-multiple-eq}, we have equivalent sets $K_1, \ldots, K_l$ to which all $\omega$-limit sets converge. We shall now define a discrete-time Markov chain of states at hitting times of neighborhoods of these compact sets. In order to bound the transition probabilities of this chain, recall the definitions of $\tilde{V}(K_i, K_j)$ given in (\ref{eqn:V-tilde}) and $V(K_i,K_j)$ given in (\ref{eqn:V-KiKj}).

Define the following quantities:
\begin{itemize}
  \item An $r_0$ such that $0 < r_0 < (1/2) \min_{i,j} \rho_0 (K_i, K_j)$,
  \item An $r_1$ such that $0 < r_1 < r_0$,
  \item The set $C$ as $C := \MA_1(\ZA) \setminus \cup_{i = 1}^l [K_i]_{r_0}$,
  \item The set $\Gamma_i$ as $\Gamma_i := \overline{[K_i]_{r_0}}$,
  \item The set $g_i$ as $g_i := [K_i]_{r_1}$, and finally,
  \item The set $g$ as $g := \cup_{i=1}^l g_i$.
\end{itemize}
Let us now define the following stopping times:
\begin{itemize}
  \item $\tau_0 := 0$,
  \item The time for exit from the union of the $r_0$ neighborhoods of the compact sets $K_i$'s, that is, $\sigma_n := \inf \{ t \geq \tau_n ~|~ \mu_N(t) \in C \}$,
  \item The time to re-enter $\overline{g}$, that is, $\tau_n := \inf \{ t \geq \sigma_n ~|~ \mu_N(t) \in \overline{g} \}$.
\end{itemize}
Finally, we define $Z_n := \mu_N(\tau_n)$. We shall use the notation $p^{(N)}(\nu, \overline{g_j})$ for $p^{(N)}_{\nu}(\mu_N(\tau_n) \in \overline{g_j})$ when $\mu_N(\tau_{n-1}) = \nu$.

\begin{lemma}{\em (Freidlin and Wentzell \cite[Ch.6, Lemma 2.1]{Freidlin})}.
For any $\varepsilon > 0$, there is a small enough $r_0 > 0$ such that for any $r_2$ satisfying $0 < r_2 < r_0$, there is an $r_1$ satisfying $0 < r_1 < r_2$ such that for all sufficiently large $N$, for all $\nu \in \overline{[K_i]_{r_2}}$, the one-step transition probabilities of $Z_n$ satisfy
\[
  \exp \{ -N (\tilde{V}(K_i, K_j) + \varepsilon) \} \leq p^{(N)}(\nu, \overline{g_j}) \leq \exp \{ -N (\tilde{V}(K_i, K_j) - \varepsilon) \}.
\]
\end{lemma}

\begin{proof}
For pairs with $\tilde{V}(K_i,K_j) = +\infty$, there is no smooth curve from $K_i$ to $K_j$ without touching one of the other compact sets. It follows that for any arbitrary $0 < r_1 < r_2 < r_0$, for all sufficiently large $N$, there is no path in $\MA_1^{(N)}(\ZA)$ from $[K_i]_{r_2}$ to $\overline{g}_j$ without touching $[K_{i'}]_{r_0}$ for some $i' \neq i, j$. Thus, for all sufficiently large $N$, we have $p^{(N)}(\nu, \overline{g}_j) = 0$. The validity of the lemma is obvious for such pairs.

For all other pairs $\tilde{V}(K_i, K_j) \leq V_0$ for some $V_0 < +\infty$.

Let us first argue the lower bound. Fix $\varepsilon > 0$. Choose $\delta$ as in part 3 of Lemma \ref{lem:S_T-bounding} so that with $\varepsilon_1 = \varepsilon/(10 C_2)$, any two points $\delta$-close can be connected by a path of duration $\varepsilon_1$ and cost at most $\varepsilon/10$. Set
\[
  r_0 = \min \{ \delta/2, (1/3)\min_{i,j} \rho_0(K_i,K_j) \}.
\]
Fix arbitrary $r_2$ satisfying $0 < r_2 < r_0$.

For each $(i,j)$ with $\tilde{V}(K_i, K_j) < +\infty$, choose paths $\mu^{i,j} : [0, T_{i,j}] \ra \MA_1(\ZA)$ such that
\begin{itemize}
  \item $\mu^{i,j}(0) \in K_i$,
  \item $\mu^{i,j}(T_{i,j}) \in K_j$,
  \item $\mu^{i,j}(t)$ does not touch $\cup_{i' \neq i,j} K_{i'}$ for $t \in [0, T_{i,j}]$, and
  \item $S_{[0,T]}(\mu^{i,j} | \mu^{i,j}(0)) \leq \tilde{V}(K_i, K_j) + 0.2 \varepsilon$.
\end{itemize}
Now choose $r_1$ so that
\begin{eqnarray*}
  r_1 \hspace*{-.1in} & < & \hspace*{-.1in} \min \Big\{ r_2, \frac{r_0}{2}, \frac{1}{2}\min \Big\{ \rho_0 \left( \mu^{i,j}(t), \cup_{i' \neq i,j} K_{i'} \right) ~|~ t \in [0, T_{i,j})], 1 \leq i,j \leq l \Big\} \Big\}.
\end{eqnarray*}
Also choose $\delta' < \min \{ r_0 - r_2, r_1 \}$ so that $r_2 + \delta' < r_0$ and hence $r_0 + \delta' < 2r_0 < \delta$ by the choice of $\delta$.

Take any $\nu \in \overline{[K_i]_{r_2}}$. Fix a finite $\delta'$-net of $K_i$. If $i \neq j$, consider the following path.
\begin{itemize}
  \item Connect $\nu$ to the nearest point $\nu_1 \in K_i$ with a path of duration $\varepsilon_1$ and cost at most $0.1 \varepsilon$.
  \item Connect $\nu_1$ to the nearest point $\nu_2$ on the $\delta'$-net of $K_i$ again with a path of duration $\varepsilon_1$ and cost at most $0.1 \varepsilon$.
  \item Let $\nu_3$ be the point on the $\delta'$-net nearest to $\mu^{i,j}(0)$. Traverse the path given by Lemma \ref{lem:FW-1.6} that connects $\nu_2$ to $\nu_3$ without leaving the $r_1$-neighborhood of $K_i$. Thanks to the finite number of points on the $\delta'$-net, this can be done in bounded time. Moreover, the cost is at most $0.1 \varepsilon$.
  \item Connect $\nu_3$ to $\mu^{i,j}(0)$ with path of duration $\varepsilon_1$ and cost at most $0.1 \varepsilon$.
  \item Then traverse the path given by $\mu^{i,j}$.
\end{itemize}
If $i = j$, then simply take $\nu$ to a point at a distance $r_0 + \delta'$ from $K_i$ and then to the nearest point in $K_i$. Note that $r_0 + \delta' < \delta$, and so the duration of this path is $2\varepsilon_1$ and cost at most $0.2 \varepsilon$. The constructed path is of bounded time duration, bounded say by $T_0$. We can thus extend all paths to duration $T_0$ along the McKean-Vlasov path, an appendage that incurs no additional cost. Call the resulting path $\mu(t), t \in [0,T_0]$. Clearly,
\[
  S_{[0,T_0]} (\mu | \nu) \leq \tilde{V}(K_i, K_j) + 0.6 \varepsilon.
\]

If $\rho_{T_0}(\mu_N, \mu) < \delta'$, then the trajectory $\mu_N$ begins at a point $\nu$ within an $r_0$-neighborhood of $K_i$, reaches the $\delta'$-neighborhood of $K_j$ and so hits $\overline{g}_j = \overline{[K_j]_{r_1}}$, is at most $\delta'$ distance away from the trajectory $\mu$, and hence does not hit the $r_0$-neighborhood of any $K_{i'}, i' \neq i,j$; then $\mu_N(\tau_n) \in \overline{g}_j$. In other words,
\[
  \{ \rho_{T_0} (\mu_N, \mu) < \delta' \} \subset \{ \mu_N(\tau_n) \in \overline{g}_j \},
\]
and so
\begin{eqnarray*}
  p^{(N)}(\nu, \overline{g}_j) & \geq &  p^{(N)}_{\nu} \{ \rho_{T_0} (\mu_N, \mu) < \delta' \} \\
  & \geq & \exp \{ -N (S_{[0,T_0]}(\mu | \nu) + 0.1 \varepsilon) \} \\
  & \geq & \exp \{ -N (\tilde{V}(K_i, K_j) + \varepsilon ) \},
\end{eqnarray*}
where the second inequality holds for all sufficiently large $N$ uniformly over the initial condition, thanks to (\ref{eqn:uniform-ldp-lb}) of Corollary \ref{cor:uniform-ldp-flows}. This establishes the lower bound.

We now prove the upper bound. Consider any path $\mu$ of some duration $T$ starting at $\nu$ in the $r_1$-neighborhood of $K_i$, ending at a point say $\xi$ in the $\delta'$-neighborhood of $\overline{g}_j$ at time $T$, and not touching any of the other compact sets $K_{i'}, i' \neq i,j$. By the choices of $r_1$ and $\delta'$, there are short paths from a point $\nu' \in K_i$ to $\nu$ and from $\xi$ to a point $\xi' \in K_j$, each of duration $\varepsilon_1$ and cost at most $0.1 \varepsilon$. The path that traverses from $\nu'$ to $\nu$, and then along $\mu$ to $\xi$, and thence to $\xi'$, has cost at most $S_{[0,T]}(\mu | \nu) + 0.2 \varepsilon \geq \tilde{V}(K_i, K_j)$, and so
\begin{equation}
  \label{eqn:lower-level-set}
  S_{[0,T]}(\mu | \nu) \geq \tilde{V}(K_i, K_j) - 0.2 \varepsilon.
\end{equation}
The same holds for any path $\mu$ of some duration $T$ starting at $\nu$ in the $r_1$-neighborhood of $K_i$, {\em touching} the $\delta'$-neighborhood of $\overline{g}_j$ at time in $[0,T]$, but not touching any of the other compact sets $K_{i'}, i' \neq i,j$.

By Lemma \ref{lem:FW-1.9}, with the set $C$ in place of $K$, a set that does not contain any $\omega$-limit set entirely, and with $T_1 = T_0 + V_0/c$ where $c,T_0$ are as specified in that lemma, we obtain
\begin{equation}
  \label{eqn:tau1-not-too-large}
  p^{(N)}_{\nu} \{ \tau_1 > T_1 \} \leq \sup_{\nu' \in C} p^{(N)}_{\nu'} \{ \tau_{C} > T_1 \} \leq e^{-N V_0}
\end{equation}
for all sufficiently large $N$.

Consider a trajectory $\mu_N$ with $\mu_N(0) = \nu \in \overline{[K_i]_{r_1}}$ and $\mu_N(\tau_1) \in \overline{g}_j$. There are two possibilities: (1) $\tau_1 > T_1$, or (2) $\tau_1 \leq T_1$ in which case the trajectory enters $\overline{g}_j$ in $[0,T_1]$. In this second case, with
\[
  \Phi_{[0,T_1], \nu}(v) := \{ \mu : [0,T_1] \ra \MA_1(\ZA) ~|~ S_{[0,T]}(\mu | \nu) \leq v \},
\]
we have
\begin{equation}
  \label{eqn:trajectory-is-far}
  \rho_{T_1} (\mu_N, \Phi_{[0,T_1], \nu}(\tilde{V}(K_i, K_j) - 0.3 \varepsilon)) \geq \delta'.
\end{equation}
To see this, note the conditions $\delta' < r_1$, $\tau_1 \leq T_1$, and $\mu_N(\tau_1) \in \overline{g}_j$. If $\mu$ is any trajectory satisfying $\rho_{T_1}(\mu_N, \mu) < \delta'$, then $\mu$ must hit the $\delta'$-neighborhood of $\overline{g}_j$ without touching any of the other compact sets $K_{i'}, i' \neq i,j$. From (\ref{eqn:lower-level-set}), subtracting an extra $0.1 \varepsilon$, we get $S_{[0,T_1]}(\mu | \nu) > \tilde{V}(K_i,K_j) - 0.3 \varepsilon$. By contraposition, under the noted conditions, any $\mu$ with $S_{[0,T_1]}(\mu | \nu) \leq \tilde{V}(K_i,K_j) - 0.3 \varepsilon$ must satisfy $\rho_{T_1}(\mu_N, \mu) \geq \delta'$, and hence (\ref{eqn:trajectory-is-far}) follows.

Putting the two cases together, we get
\begin{eqnarray*}
  p^{(N)}_{\nu} \{ \mu_N(\tau_1) \in \overline{g}_j \}
  & \leq & p^{(N)}_{\nu} \{ \tau_1 > T_1 \} \\
  & & +~ p^{(N)}_{\nu} \{ \rho_{T_1} (\mu_N, \Phi_{[0,T_1], \nu}(\tilde{V}(K_i, K_j) - 0.3 \varepsilon)) \geq \delta' \} \\
  & \stackrel{(a)}{\leq} & e^{-NV_0} + \exp \{ -N ( \tilde{V}(K_i, K_j) - 0.3 \varepsilon) + N (0.1 \varepsilon) \} \\
  & \stackrel{(b)}{\leq} & \exp \{ -N ( \tilde{V}(K_i, K_j) - \varepsilon) \}.
\end{eqnarray*}
In the above sequence of inequalities, (a) holds for all sufficiently large $N$ due to (\ref{eqn:tau1-not-too-large}), (\ref{eqn:uniform-ldp-ub}) of Corollary \ref{cor:uniform-ldp-flows}, and the definition of $\Phi_{[0,T_1], \nu}$. Inequality (b) also holds for all sufficiently large $N$ because $\tilde{V}(K_i, K_j) \leq V_0$. This proves the upper bound and completes the proof.
\end{proof}

Recall the definition $\mathbb{G}\{i\}$ in the paragraph preceding (\ref{eqn:W-compactset}), the definition of $W(K_i)$ in (\ref{eqn:W-compactset}), and the definition of $s_i, i = 1, \ldots, l$. We are now ready to state the main theorem of this appendix.

\begin{thm}{\em (Freidlin and Wentzell \cite[Ch.6, Theorem 4.1]{Freidlin})}.
\label{thm:FW-6.4.1}
Assume {\em ({\bf A1})}-{\em ({\bf A3})} and {\em ({\bf B})} hold. For any $\varepsilon > 0$, there is an arbitrarily small $r_1 > 0$ such that
\begin{eqnarray*}
  \exp \{ -N (s_i + \varepsilon) \} \leq \wp^{(N)} \{ [K_i]_{r_1} \} \leq \exp \{ -N (s_i - \varepsilon) \}
\end{eqnarray*}
where $s_i$ are defined in (\ref{eqn:s-values}).
\end{thm}

\begin{proof}
All the steps of the proof of \cite[Ch.6, Th.4.1]{Freidlin} hold, since the analogs of all the lemmas used in that proof have now been verified to hold.
\end{proof}
\end{appendix}

\section*{Acknowledgements}
We thank an anonymous reviewer for suggestions that helped improve the presentation.

\end{document}